\documentclass[12pt]{article}

\usepackage{fullpage,amsmath,amssymb,amsthm}

\usepackage{hyperref}

\usepackage{tikz}

\newtheorem{theorem}{Theorem}[section]
\newtheorem{lemma}[theorem]{Lemma}
\newtheorem{prop}[theorem]{Proposition}
\newtheorem{claim}[theorem]{Claim}
\newtheorem{definition}[theorem]{Definition}
\newtheorem{construction}[theorem]{Construction}
\newtheorem{observation}[theorem]{Observation}

\newcommand{\cH}{\mathcal{H}}
\newcommand{\cL}{\mathcal{L}}
\newcommand{\cF}{\mathcal{F}}
\newcommand{\cG}{\mathcal{G}}
\newcommand{\ex}{\mathrm{ex}}

\newcommand{\ES}{Erd\H{o}s-Simonovits}

\newcommand{\floor}[1]{\left\lfloor#1\right\rfloor}
\newcommand{\ceil}[1]{\left\lceil#1\right\rceil}

\usepackage{xcolor}

\title{Rational exponents for cliques}
\author{%
Sean English\footnote{University of North Carolina Wilmington. Contact: \texttt{EnglishS@uncw.edu}.}%
\and%
Anastasia Halfpap\footnote{Iowa State University. Contact: \texttt{ahalfpap@iastate.edu}.}%
\and%
Robert A.\ Krueger\footnote{Carnegie Mellon University. Contact: \texttt{rkrueger@andrew.cmu.edu}. Supported by NSF Awards DGE 21-4675 and DMS-2402204.}}
\date{}

\begin{document}
	
	\maketitle
	
	\begin{abstract}
		Let $\ex(n,H,\cF)$ be the maximum number of copies of $H$ in an $n$-vertex graph which contains no copy of a graph from $\cF$. Thinking of $H$ and $\cF$ as fixed, we study the asymptotics of $\ex(n,H,\cF)$ in $n$. We say that a rational number $r$ is \emph{realizable for $H$} if there exists a finite family $\cF$ such that $\ex(n,H,\cF) = \Theta(n^r)$. Using randomized algebraic constructions, Bukh and Conlon showed that every rational between $1$ and $2$ is realizable for $K_2$. We generalize their result to show that every rational between $1$ and $t$ is realizable for $K_t$, for all $t \geq 2$. We also determine the realizable rationals for stars and note the connection to a related Sidorenko-type supersaturation problem.
	\end{abstract}
	
	\section{Introduction}\label{sec::intro}
	
	For a family of graphs $\cF$ and a graph $H$, let $\ex(n,H,\cF)$ be the maximum number of copies of $H$ in an $n$-vertex $\cF$-free graph. A central object of study in extremal graph theory is $\ex(n,K_2,\cF)$, dating back to the earliest theorems of Mantel~\cite{M07} and Tur\'an~\cite{T41}. While $\ex(n,H,\cF)$ had been studied sporadically for particular $H$ and $\cF$, a systematic general study was started by Alon and Shikhelman~\cite{AS16}, and this extremal function has received much attention since then (see for instance \cite{GGMV2020, GP2019, GP2022, HP2020, MQ2018, Morrison2022}). We are interested in the asymptotics in $n$ of $\ex(n,H,\cF)$, thinking of $H$ and $\cF$ as fixed and finite; in particular, we are interested in \emph{which} asymptotics are possible for a given $H$.
	
	The Erd\H{o}s--Stone--Simonovits Theorem~\cite{ES46,ES66} states that
    \[ \ex(n,K_2,\cF) = \left(1 - \frac{1}{\chi(\cF)-1} + o(1) \right)\binom{n}{2} ,\]
    where $\chi(\cF)$ is the smallest chromatic number of any graph of $\cF$, satisfactorily determining the asymptotics of $\ex(n,K_2,\cF)$ when $\chi(\cF) \geq 3$. It is a major open problem in extremal graph theory to determine the asymptotics of $\ex(n,K_2,\cF)$ when $\cF$ contains a bipartite graph, and many natural open problems and conjectures have arisen around attempts to understand these asymptotics. Erd\H{o}s and Simonovits made a number of natural conjectures, which have been reiterated by others (see the survey of F\"uredi and Simonovits~\cite{FS13}). One conjecture~\cite{E81} is that for every $F$, $\ex(n,K_2,F) = (c_F+o(1)) n^{r_F}$ for some constants $c_F$ and $r_F$ depending on $F$; furthermore, they asked if $r_F$ is always rational.\footnote{We note that this conjecture is false when $K_3$ is substituted for $K_2$: Ruzsa and Szemer\'edi's results on the $(6,3)$-problem show that $\ex(n,K_3,K_4^-)$ is $o(n^2)$ but at least $n^{2-o(1)}$~\cite{AS16, RS78}. Analogously, the hypergraph analogue of this conjecture is false.} Inversely, they also asked~\cite{E81} if the following is true: for every rational $r \in [1,2]$, there exists a graph $F_r$ such that $\ex(n,K_2,F_r) = \Theta(n^r)$, or more weakly, that there exists a finite family $\cF_r$ such that $\ex(n,K_2,\cF_r) = \Theta(n^r)$. As usual, $f(n) = \Theta(g(n))$ means that there exists constants $c_1, c_2 > 0$ such that $c_1 g(n) \leq f(n) \leq c_2 g(n)$ for sufficiently large $n$. After several decades, this last conjecture was confirmed by Bukh and Conlon~\cite{BC18}.
	
	\begin{theorem}[Bukh, Conlon~\cite{BC18}]\label{thm::BukhConlon}
		For every $r \in [1,2]$, there exists a finite family $\cF_r$ such that $\ex(n,K_2,\cF_r) = \Theta(n^r)$.
	\end{theorem}
	
	Following~\cite{BC18}, there has been much work~\cite{CO2022, CLJ2021, JJM22, JMY22, JQ23, JQ2020, KKL21} to replace $\cF_r$ with a single graph in Theorem~\ref{thm::BukhConlon}. In this work, we instead generalize Theorem~\ref{thm::BukhConlon} to counting graphs other than $K_2$. We call a rational number $r$ \emph{realizable for $H$} if there exists a finite family $\cF_r$ such that $\ex(n,H,\cF_r) = \Theta(n^r)$. It is easy to show that, for any $H$, rationals strictly between $0$ and $1$, and those greater than $|V(H)|$, are not realizable for $H$, and hence Theorem~\ref{thm::BukhConlon} states that the realizable rationals for $K_2$ are $0$ and all those between $1$ and $2$, inclusive. Our main result is an extension of Theorem~\ref{thm::BukhConlon} from $K_2$ to $K_t$ for $t \geq 3$.
	
	\begin{theorem}\label{thm::Ktrealizable}
		Every rational between $1$ and $t$ is realizable for $K_t$; that is, for every rational $r \in [1,t]$, there exists a finite family of graphs $\cF_r$ such that $\ex(n,K_t,\cF_r) = \Theta(n^r)$.
	\end{theorem}
	
	Perhaps surprisingly, there exist graphs $H$ for which not every exponent between $1$ and $|V(H)|$ is realizable. Let $S_t$ be the star with $t$ edges.
	
	\begin{theorem}\label{thm::starsrealizable}
		The rationals realizable for $S_t$ are precisely $0$, $1$, and those in $[t,t+1]$.
	\end{theorem}
	
	In order to show that $r$ is realizable for $H$, we must create a family $\mathcal{F}_r$ and demonstrate upper and lower bounds on $\ex(n,H,\cF_r)$. The lower bound construction is a graph based on random polynomials, and is essentially the same as the method used by Bukh and Conlon; see Section~\ref{sec::intro:lower}. Bukh and Conlon's upper bound argument for $K_2$ was relatively simple, but most of our work is in extending this argument from $K_2$ to other $H$; we must construct families $\cF$ that satisfy the technical conditions of our arguments while having enough variety so as to realize every desired rational. See Section~\ref{sec::intro:upper} for discussion of our upper bound, which may be of independent interest for its relation to Sidorenko-type generalized supersaturation problems.
	
	Before moving to the proof techniques, we note that Fitch~\cite{F19} proved an analogue of Theorem~\ref{thm::Ktrealizable} for hypergraphs (extending an older theorem of Frankl~\cite{F86}): for every integer $k$ and rational $r\in [1, k]$, there exists a finite family of $k$-uniform hypergraphs $\cF_{k,r}$ such that the maximum number of hyperedges in a $k$-uniform $\cF_{k,r}$-free hypergraph on $n$ vertices is $\Theta(n^r)$. Fitch's techniques are similar to ours and Bukh and Conlon's, although we run into some additional complications --- Fitch only shows that rationals in $[k-1,k]$ are realizable in the $k$-uniform setting, and achieves all other rationals through a stepping-up argument from lower uniformities that is unavailable to us. In particular, the natural attempt of taking the $t$-uniform constructions and replacing each hyperedge with a copy of $K_t$ does not work. We also hope our presentation gives some indication on how to proceed when $H$ is not a clique, although we do not pursue this direction. We view our Theorem~\ref{thm::Ktrealizable} as a natural extension of Theorem~\ref{thm::BukhConlon} and Fitch's result with the recent study of generalized extremal numbers.

	\subsection{Lower bounds: the form of our forbidden families}\label{sec::intro:lower}
	
	Bukh and Conlon proved Theorem~\ref{thm::BukhConlon} using \emph{powers} of \emph{balanced} \emph{rooted trees}; we define the relevant terminology now.
	
	\begin{definition}\label{def::rooted graph} \rm
		A \emph{rooted graph} $(F,R)$ is a graph $F$ with a specified set $R \subseteq V(F)$ of \emph{rooted vertices}. When the set of roots is understood, we often conflate $(F,R)$ with $F$. The \emph{$\ell^{\text{th}}$ power of $(F,R)$}, denoted $(F,R)^\ell$, is the set of all graphs obtained as the union of $\ell$ distinct copies of $F$ which share $R$ and which may intersect in any way outside of $R$.
	\end{definition}
	
\begin{figure}[ht]
		\begin{center}
			\begin{tikzpicture}[scale = 0.8]

				\filldraw (0,0) circle (0.05 cm);
				\filldraw (1,0) circle (0.05 cm);
				\filldraw (2,0) circle (0.05 cm);
				\filldraw (3,0) circle (0.05 cm);
				\filldraw (4,0) circle (0.05 cm);
				\filldraw (0,-1) circle (0.05 cm);
				\filldraw (2,-1) circle (0.05 cm);
				\filldraw (4,-1) circle (0.05 cm);
				\draw (0,0) -- (4,0);
				\draw (0,0) -- (0,-1);
				\draw (2,0) -- (2,-1);
				\draw (4,0) -- (4,-1);
				\draw (-0.5,-.7) -- (4.5,-.7);

                \draw (5.5, -.7) -- (10.5,-.7);
				\filldraw (6,0) circle (0.05 cm);
				\filldraw (7,0) circle (0.05 cm);
				\filldraw (8,0) circle (0.05 cm);
				\filldraw (9,0) circle (0.05 cm);
				\filldraw (10,0) circle (0.05 cm);
                \filldraw (6.5,0.5) circle (0.05 cm);
				\filldraw (7.5,0.5) circle (0.05 cm);
				\filldraw (8.5,0.5) circle (0.05 cm);
				\filldraw (9.5,0.5) circle (0.05 cm);
				\filldraw (10.5,0.5) circle (0.05 cm);
				\filldraw (6,-1) circle (0.05 cm);
				\filldraw (8,-1) circle (0.05 cm);
				\filldraw (10,-1) circle (0.05 cm);
                \draw (6,0) --  (6, -1);
				\draw (8,0) --  (8, -1);
				\draw (10,0)  -- (10, -1);
				\draw (6,0)  -- (7,0);
				\draw (7,0)  -- (8,0);
				\draw (8,0)  -- (9,0);
				\draw (9,0)  -- (10,0);
                \draw (6.5,0.5) -- (6,-1);
                \draw (8.5, 0.5) -- (8,-1); 
                \draw (10.5, 0.5) -- (10,-1);
                \draw (6.5,0.5) -- (10.5,0.5);
                \filldraw (7,1) circle (0.05 cm);
				\filldraw (8,1) circle (0.05 cm);
				\filldraw (9,1) circle (0.05 cm);
				\filldraw (10,1) circle (0.05 cm);
				\filldraw (11,1) circle (0.05 cm);
                \draw (6,-1) -- (7,1);
                \draw (8,-1) -- (9,1);
                \draw (10,-1) -- (11,1);
                \draw (7,1) -- (11,1);

                \draw (12.5, -.7) -- (17.5,-.7);
                \draw (13,0) --  (13, -1);
				\draw (15,0) --  (15, -1);
				\draw (17,0)  -- (17, -1);
				\draw (13,0)  -- (14,0);
				\draw (14,0)  -- (15,0);
				\draw (15,0)  -- (16,0);
				\draw (16,0)  -- (17,0);
                \draw[blue, very thick] (13.5,0.5) -- (13,-1);
                \draw[blue, very thick] (15.5, 0.5) -- (15,-1); 
                \draw[blue, very thick] (17.5, 0.5) -- (17,-1);
                \draw[blue, very thick] (13.5,0.5) -- (14,0);
                \draw[blue, very thick] (14,0) -- (15.5,0.5);
                \draw[blue, very thick] (15.5,0.5) -- (17,0);
                \draw[blue, very thick] (17,0) -- (17.5,0.5);

                \draw[red, very thick, dashed] (13,-1) -- (14,1);
                \draw[red, very thick, dashed] (15,-1) -- (16,1);
                \draw[red, very thick, dashed] (17,-1) -- (18,1);
                \draw[red, very thick, dashed] (14,1) -- (15.5, 0.5);
                \draw[red, very thick, dashed] (15.5,0.5) -- (16,1);
                \draw[red, very thick, dashed] (16,1) -- (16,0);
                \draw[red, very thick, dashed](16,0) -- (18,1);

                \filldraw (13,0) circle (0.05 cm);
				\filldraw (14,0) circle (0.05 cm);
				\filldraw (15,0) circle (0.05 cm);
				\filldraw (16,0) circle (0.05 cm);
				\filldraw (17,0) circle (0.05 cm);
                \filldraw (13.5,0.5) circle (0.05 cm);
				\filldraw (15.5,0.5) circle (0.05 cm);
				\filldraw (17.5,0.5) circle (0.05 cm);
				\filldraw (13,-1) circle (0.05 cm);
				\filldraw (15,-1) circle (0.05 cm);
				\filldraw (17,-1) circle (0.05 cm);
                 \filldraw (14,1) circle (0.05 cm);
				\filldraw (16,1) circle (0.05 cm);
				\filldraw (18,1) circle (0.05 cm);

			\end{tikzpicture}
			
			\caption{A rooted tree $T$ (left) and two graphs in the $3^{\text{rd}}$ power of $T$ (middle and right). On the right, we use distinct colors and line weights to distinguish the three overlapping copies of $T$. Rooted vertices are below the line.}\label{tree powers}    
		\end{center}
	\end{figure}
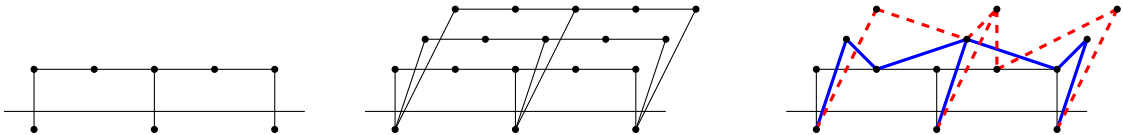
	
	The main parameter controlling the extremal number of $(F,R)^\ell$ is the \emph{rooted density}, as defined below.
	
	\begin{definition}\label{def::rooted density} \rm
		Let $(F,R)$ be a rooted graph. The \emph{rooted density} of $S \subseteq V(F) \setminus R$, is the number of edges incident to $S$ divided by $|S|$, denoted $d(S) = e(S)/|S|$. The rooted density of $(F,R)$ is the rooted density of $V(F) \setminus R$, denoted $d((F,R)) = d(V(F) \setminus R)$. We say that $(F,R)$ is \emph{balanced} if $d(S)$ is minimized by $S = V(F) \setminus R$.
	\end{definition}
	
	Our main lower bounds come in the following general form, proven by Spiro~\cite{S24} as a straightforward modification of the random polynomial construction of Bukh and Conlon~\cite{BC18}.
	
	\begin{theorem}[Spiro~\cite{S24}]\label{thm::lowerbound}
		Let $(F_i,R_i)$ be a finite number of balanced rooted graphs of rooted density at least $d$, and let $H$ be a graph. Then there exists $L_0$ such that for all $L \geq L_0$, $\ex(n,H,\cF) = \Omega(n^{|V(H)|-|E(H)|/d})$, where
		\[ 
  \cF = \bigcup_i (F_i,R_i)^L.
  \]
	\end{theorem}
	
	Thus to achieve a particular exponent for $H$, we find the appropriate density $d$ and then engineer a set of balanced rooted graphs each with that rooted density. We require the $(F_i, R_i)$ to be balanced so that every rooted graph in $(F_i,R_i)^\ell$ (with root set $R$) has rooted density at least as large as the rooted density of $(F_i,R_i)$.
	
	Our constructions utilize the rooted graph construction Bukh and Conlon use to prove Theorem~\ref{thm::BukhConlon}; we recall this construction and its main properties now.
	
	\begin{construction}[Bukh, Conlon~\cite{BC18}]\label{constr::BC}
		Let $b > a \geq 1$ be integers, and let $T_{K_2}(a,b)$ denote the following rooted tree. Begin with a path $u_1, \dots, u_a$, which is not in the root and which collectively we call the \emph{spine}, and attach a new leaf, which is a root, to $u_i$ each time $i$ appears in the following sequence:
		\[ \floor{1 + 0 \cdot \frac{a}{b-a}}, \floor{1 + 1 \cdot \frac{a}{b-a}}, \dots, \floor{1 + (b-a-1) \cdot \frac{a}{b-a}}, a .\]
		See Figure~\ref{bktrees}.
	\end{construction}

	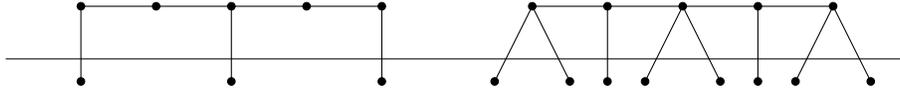
\begin{figure}[ht]
		\begin{center}
			\begin{tikzpicture}

				\filldraw (0,0) circle (0.05 cm);
				\filldraw (1,0) circle (0.05 cm);
				\filldraw (2,0) circle (0.05 cm);
				\filldraw (3,0) circle (0.05 cm);
				\filldraw (4,0) circle (0.05 cm);
				\filldraw (0,-1) circle (0.05 cm);
				\filldraw (2,-1) circle (0.05 cm);
				\filldraw (4,-1) circle (0.05 cm);
				\draw (0,0) -- (4,0);
				\draw (0,0) -- (0,-1);
				\draw (2,0) -- (2,-1);
				\draw (4,0) -- (4,-1);
				\draw (-1,-.7) -- (11,-.7);
				
				\filldraw (6,0) circle (0.05 cm);
				\filldraw (7,0) circle (0.05 cm);
				\filldraw (8,0) circle (0.05 cm);
				\filldraw (9,0) circle (0.05 cm);
				\filldraw (10,0) circle (0.05 cm);
				\filldraw (6.5,-1) circle (0.05 cm);
				\filldraw (8.5,-1) circle (0.05 cm);
				\filldraw (10.5,-1) circle (0.05 cm);
    			\filldraw (5.5,-1) circle (0.05 cm);
				\filldraw (7.5,-1) circle (0.05 cm);
				\filldraw (9.5,-1) circle (0.05 cm);
                \filldraw (7,-1) circle (0.05 cm);
                \filldraw (9,-1) circle (0.05 cm);
				\draw (6,0) --  (5.5, -1);
				\draw (8,0) --  (7.5, -1);
				\draw (10,0)  -- (9.5, -1);
                \draw (6,0) --  (6.5, -1);
				\draw (8,0) --  (8.5, -1);
				\draw (10,0)  -- (10.5, -1);
				\draw (6,0)  -- (7,0);
				\draw (7,0)  -- (8,0);
				\draw (8,0)  -- (9,0);
				\draw (9,0)  -- (10,0);
                \draw (7,0) -- (7,-1);
                \draw (9,0) -- (9,-1);
			\end{tikzpicture}
			
			\caption{$T_{K_2}(5,7)$ (left) and $T_{K_2}(5,12)$ (right). Rooted vertices are below the line.}\label{bktrees}    
		\end{center}
	\end{figure}

	\begin{prop}[Bukh, Conlon~\cite{BC18}]\label{prop::balanced:BC}
		For all $b > a \geq 1$, $T_{K_2}(a,b)$ is balanced with rooted density $b/a$. In particular, every non-root vertex of $T_{K_2}(a,b)$ has degree at least $b/a$.
	\end{prop}

	\subsection{Upper bounds: a supersaturation problem}\label{sec::intro:upper}
	
	Theorem~\ref{thm::lowerbound} provides a lower bound on $\ex(n,H,\cF)$, as long as $\cF$ is the union of sufficiently large powers of balanced rooted graphs, $(F,R)^\ell$. The variety and simplicity of $(F,R)^\ell$ makes the upper bound arguments have a simple structure: if there are more than $\ell-1$ times as many copies of $F$ in a graph $G$ as copies of $F[R]$, then by the pigeonhole principle, we can find an element of $(F,R)^\ell$ in $G$. We can limit the number of copies of $F[R]$ by putting powers of other balanced rooted graphs in $\cF$, assuming we understand which exponents are realizable for the components of $F[R]$. For example, for the $(F,R)$ we construct, in the case when $H$ is a clique (Theorem~\ref{thm::Ktrealizable}), $F[R]$ is the disjoint union of smaller cliques, and hence we can bound the number of copies of it using induction on the size of the clique.
	
	The main effort is in solving a supersaturation problem --- showing that there are many copies of $F$ in an $n$-vertex graph $G$ with many copies of $H$. The quantification of both instances of `many' comes from intuition provided by the random graph $G(n,p)$. Indeed, the graphs constructed in the proof of Theorem~\ref{thm::lowerbound} have the same \emph{small} subgraph counts as $G(n,p)$ for appropriate $p$. Erd\H{o}s and Simonovits~\cite{ES84} and Sidorenko~\cite{S91,S93} conjecture that, for bipartite $F$, the random graph approximately minimizes the number of copies of $F$ among all graphs with a given number of edges. This supersaturation approach is fundamental to degenerate extremal graph theory problems, appearing as early as K\"ov\'ari, S\'os, and Tur\'an's bound on $\ex(n,K_2,K_{s,t})$~\cite{KST54}; recently, supersaturation has proven an important ingredient in many asymptotic enumeration results using the hypergraph container method of~\cite{BMS18, ST2015}. Generalized supersaturation problems, where we assume $G$ contains many copies of $H$ and conclude that $G$ has many copies of $F$, have been recently systematically studied (see, e.g., \cite{CNR2022, GNV2022, HP2020}). 
	
	Call $F$ \emph{$H$-\ES}~if there exists $\alpha > 0$ and $c>0$ such that for every $p \geq n^{2-\alpha}$ and $n$ sufficiently large, every $n$-vertex graph with at least $n^{|V(H)|} p^{|E(H)|}$ copies of $H$ contains at least $c n^{|V(F)|} p^{|E(F)|}$ copies\footnote{The related notion of a graph $F$ being \emph{Sidorenko} has to do with graph homomorphisms, rather than copies, of $F$ in $n$-vertex graphs with $n$ sufficiently large. In that case of Sidorenko's conjecture~\cite{S91,S93}, one takes $c$ to be $1-o(1)$.} of $F$. Erd\H{o}s and Simonovits's conjecture is that all bipartite graphs are $K_2$-\ES.\footnote{It is not easy to conjecture which graphs are $K_t$-\ES. For example, $K_{2,2,2}$ is not $K_3$-\ES~\cite{AS16}, but Dubroff, Gunby, Narayanan, and Spiro~\cite{DGNS23} (roughly) conjecture that $K_{3,t}$ is $K_3$-\ES.} For example, trees are $K_2$-\ES~(for every $\alpha < 1$), because every $n$-vertex graph with $n^2 p$ edges has a subgraph of minimum degree at least $np$, in which we can greedily find many copies of the desired tree. This simple argument proves Theorem~\ref{thm::BukhConlon} when combined with Theorem~\ref{thm::lowerbound} using $T_{K_2}(a,b)$ from Construction~\ref{constr::BC}, observing that the root sets of these graphs form independent sets. To prove Theorem~\ref{thm::Ktrealizable}, we generalize this argument, showing that $K_t$-trees, graphs `built' out of copies of $K_t$ in a tree-like manner, are $K_t$-\ES. Our constructions take the following form.
	
	\begin{definition}\label{def::Htree}
		Given graphs $H$ and $T$, we say $T$ is an $H$-tree if there exists a sequence of subgraphs of $T$ isomorphic to $H$, $H_1\subseteq T, H_2\subseteq T, \dots, H_b\subseteq T$, with $T=\bigcup_{i=1}^b H_i$, and with the property that, for every $2\leq i\leq b$, there exists $j<i$ such that the following three conditions hold.
		\begin{enumerate}
			\item $V(H_i) \cap \left(V(H_1)\cup\cdots\cup V(H_{i-1})\right) = V(H_i) \cap V(H_j)$,
			\item $V(H_i) \cap V(H_j)$ is neither empty nor equal to $V(H_i)$, and
			\item there exists an isomorphism between $H_i$ and $H_j$ which fixes $V(H_i) \cap V(H_j)$.
		\end{enumerate}
	\end{definition}
	
	Section~\ref{sec::tools} develops several tools which show in generality that $K_t$-trees are $K_t$-\ES; in fact, we show that every $n$-vertex graph with at least $n^{|V(K_t)|} p^{|E(K_t)|}$ copies of $K_t$ contains $\Omega(n^{|V(F)|} p^{|E(F)|})$ copies of a $K_t$-tree $F$ for all $p$ above a certain optimal threshold, depending on the sizes of $H_i \cap H_j$ which appear in the definition of $H$-tree, assuming the graph does not contain some other forbidden subgraphs. This is needed to obtain all the rational exponents in Theorem~\ref{thm::Ktrealizable}. The main thing we are fighting is injectivity --- that is, we seek copies of $F$, not homomorphisms of $F$ into $G$. This manifests in some important technical constraints that we must work around when creating our $K_t$-tree constructions. In fact, injectivity is the challenge in much of the recent work on improving the forbidden families from Theorem~\ref{thm::BukhConlon} to single forbidden graphs (see~\cite{CO2022, CLJ2021, JJM22, JMY22, JQ23, JQ2020, KKL21}).
	
	To conclude this overview, we summarize the major difficulties we face: for every rational density, we need a balanced $H$-tree with that rooted density and a comparatively simple structure induced by its roots; furthermore, we must prove that this $H$-tree is $H$-\ES~at that density.

	\subsection{Definitions, notation and organization}\label{sec::intro:org}
	
	We will silently assume throughout the paper that no graphs $F\in \mathcal{F}$ contain isolated vertices, as we are only concerned with finite $\mathcal{F}$ and $n$ large, so isolated vertices will not affect the extremal number. Given a graph $G$ and disjoint sets $A,B\subseteq V(G)$, we define the induced bipartite graph $G[A,B]$ to be the graph with $V(G[A,B])=A\cup B$ and $E(G[A,B])=\{ab\in E(G)\mid a\in A,b\in B\}$.
	
	In Section~\ref{sec::stars} we prove Theorem~\ref{thm::starsrealizable} on the realizable exponents for stars, which serves as a gentle introduction to the ideas needed for Theorem~\ref{thm::Ktrealizable}. In Section~\ref{sec::tools} we create some general tools for these lower bound problems for cliques, which we expect could be generalized to other graphs. In Section~\ref{section constructions}, we give the rooted $K_t$-tree constructions, which are required in order to define our forbidden families, and prove that they are balanced. Finally in Section~\ref{main thm sec}, we provide the proof of Theorem~\ref{thm::Ktrealizable}.

	\section{Exponents for stars}\label{sec::stars}
	
	Recall that $S_t$ is the star with $t$ edges. We first observe that not all exponents for $S_t$ are achievable.
	
	\begin{lemma}\label{lem::star non realizable}
		Let $\cF$ be a finite family of graphs. Then $\ex(n,S_t, \cF)$ is either $\Theta(1)$, $\Theta(n)$, or $\Omega(n^t)$.
	\end{lemma}
	
	\begin{proof}
		If $\cF$ does not contain a star, then $S_{n-1}$ is $\cF$-free and contains $\Omega(n^t)$ copies of $S_t$. Otherwise, $\cF$ contains a star, so there exists $\Delta = O(1)$ such that every $\cF$-free graph must have maximum degree at most $\Delta$. This means that every $\cF$-free graph contains at most $n \binom{\Delta}{t} = O(n)$ copies of $S_t$. Furthermore, if the disjoint union of $\floor{n/t}$ copies of $S_t$ is $\cF$-free, then $\ex(n, S_t, \cF) = \Theta(n)$. Otherwise there exists $c = O(1)$ such that some element of $\cF$ is a subgraph of the disjoint union of $c$ copies of $S_t$. In this case, let $G$ be a $\cF$-free graph. Since $G$ has maximum degree at most $\Delta$, at most $(t+1) (\Delta+1) \binom{\Delta}{t}$ copies of $S_t$ intersect a given copy of $S_t$ in $G$. Thus if $G$ has more than $c (t+1)(\Delta+1) \binom{\Delta}{t}$ copies of $S_t$, then $G$ contains $c$ vertex disjoint copies of $S_t$, contradicting that $G$ is $\cF$-free, and so $\ex(n,S_t,\cF) = O(1)$.
	\end{proof}
	
	The exponents of $0$ and $1$ are achievable, by taking $\cF = \{S_t\}$ and $\cF = \{S_{t+1}\}$. Furthermore, $t+1$ and $t$ are achievable by taking $\cF = \emptyset$ and $\cF = \{2K_2\}$.\footnote{We remark that while taking $\mathcal{F} = \emptyset$ is the simplest example of a family achieving the exponent $t+1$, there are many non-empty families $\mathcal{F}$ we could also choose. Alon and Shikhelman~\cite{AS16} showed that  $\ex(n,H,F) = o(n^{|V(H)|})$ if and only if $F$ is a subgraph of a blow-up of $H$. Otherwise, $\ex(n,H,F) = \Omega(n^{|V(H)|})$. So in fact, it is quite easy to find non-trivial examples achieving an exponent of $|V(H)|$.} We now show that every exponent between $t$ and $t+1$ is achievable using the following rooted graph.
	
	\begin{construction}\label{construction stars}
		Let $b > a \geq 1$ be integers, and consider $T_{K_2}(a,b)$ with spine $u_1, \dots, u_a$. Construct $T_{S_t}(a,b)$ by appending $t-1$ new root leaves to each root vertex of $T_{K_2}(a,b)$ adjacent to a $u_i$ with odd $i$. See Figure~\ref{st trees}.
	\end{construction}

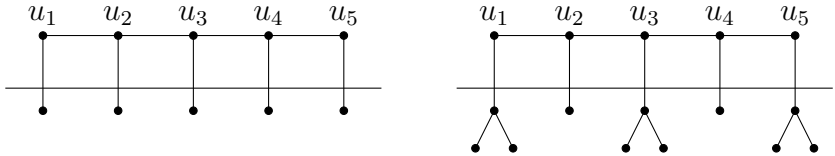
\begin{figure}[ht]

\begin{center}

\begin{tikzpicture}

				\filldraw (0,0) circle (0.05 cm) node[above]{$u_1$};
				\filldraw (1,0) circle (0.05 cm) node[above]{$u_2$};
				\filldraw (2,0) circle (0.05 cm) node[above]{$u_3$};
				\filldraw (3,0) circle (0.05 cm) node[above]{$u_4$};
				\filldraw (4,0) circle (0.05 cm) node[above]{$u_5$};
				\filldraw (0,-1) circle (0.05 cm);
                \filldraw (1,-1) circle (0.05 cm);
                \filldraw (3,-1) circle (0.05 cm);
				\filldraw (2,-1) circle (0.05 cm);
				\filldraw (4,-1) circle (0.05 cm);
				\draw (0,0) -- (4,0);
				\draw (0,0) -- (0,-1);
				\draw (2,0) -- (2,-1);
				\draw (4,0) -- (4,-1);
                \draw (1,-1) -- (1,0);
                \draw (3,-1) -- (3,0);
				\draw (-0.5,-.7) -- (4.5,-.7);

                \draw (5.5, -.7) -- (10.5,-.7);
				\filldraw (6,0) circle (0.05 cm)node[above]{$u_1$};
				\filldraw (7,0) circle (0.05 cm)node[above]{$u_2$};
				\filldraw (8,0) circle (0.05 cm)node[above]{$u_3$};
				\filldraw (9,0) circle (0.05 cm)node[above]{$u_4$};
				\filldraw (10,0) circle (0.05 cm)node[above]{$u_5$};
				\filldraw (6,-1) circle (0.05 cm);
				\filldraw (8,-1) circle (0.05 cm);
				\filldraw (10,-1) circle (0.05 cm);
                \filldraw (7,-1) circle (0.05 cm);
                \filldraw (9,-1) circle (0.05 cm);
                \draw (6,0) --  (6, -1);
				\draw (8,0) --  (8, -1);
				\draw (10,0)  -- (10, -1);
				\draw (6,0)  -- (7,0);
				\draw (7,0)  -- (8,0);
				\draw (8,0)  -- (9,0);
				\draw (9,0)  -- (10,0);
                \draw (7,0) -- (7,-1);
                \draw (9,0) -- (9,-1);

                \filldraw (5.75, -1.5) circle (0.05 cm);
                \filldraw (6.25, -1.5) circle (0.05 cm);
                \filldraw (7.75, -1.5) circle (0.05 cm);
                \filldraw (8.25, -1.5) circle (0.05 cm);
                \filldraw (9.75, -1.5) circle (0.05 cm);
                \filldraw (10.25, -1.5) circle (0.05 cm);
                \draw (5.75,-1.5) -- (6,-1) -- (6.25,-1.5);
                \draw (7.75,-1.5) -- (8,-1) -- (8.25,-1.5);
                \draw (9.75,-1.5) -- (10,-1) -- (10.25,-1.5);

\end{tikzpicture}
\end{center}

\caption{$T_{K_2}(5,9)$ and $T_{S_3}(5,9)$}\label{st trees}

\end{figure}
 
	Observe that the root set of $T_{S_t}(a,b)$ is a disjoint union of isolated vertices and copies of $S_{t-1}$. Forbidding $T_{S_{t-1}}(a,b)^\ell$ will limit the number of choices for the root by induction, and hence will ensure that if there are many copies of $T_{S_t}(a,b)$, then we will obtain some power of $T_{S_t}(a,b)$.
	
	Now we check that $T_{S_t}(a,b)$ will obtain the lower bounds we desire.
	
	\begin{lemma}
		The rooted graph $T_{S_t}(a,b)$ is balanced with rooted density $b/a$.
	\end{lemma}
	
	\begin{proof}
		Since we have not added any spine vertices or edges incident to the spine of $T_{K_2}(a,b)$ to obtain $T_{S_t}(a,b)$, the balancedness and rooted density follow immediately.
	\end{proof}
	
	The following lemma is a helpful pruning step that we use to count the number of copies of $T_{S_t}(a,b)$ in a graph with many copies of $S_t$. The perhaps strange parameterization of this lemma is meant to evoke its later use.
	
	\begin{lemma}\label{lem::star subgraph}
		Let $G$ be an $n$-vertex graph containing at least $3C n^{t+1} p^t$ copies of $S_t$. Then there exists a nonempty collection $\cH$ of copies of $S_t$ and collections $\cL_1$ and $\cL_2$ of vertices such that
		\begin{enumerate}
			\item every copy of $S_t$ in $\cH$ has its center in $\cL_1$ and its leaves in $\cL_2$, and $|\cH| \geq C n^{t+1} p^t$,
			\item every vertex in $\cL_1$ is the center of at least $C n^t p^t$ copies of $S_t$ in $\cH$, and
			\item every vertex in $\cL_2$ is the leaf of at least $C n^t p^t$ copies of $S_t$ in $\cH$.
		\end{enumerate}
	\end{lemma}
	
	\begin{proof} We will find $\cH$, $\cL_1$ and $\cL_2$ via a simple algorithm. Initialize $\cH$ to be the set of all copies of $S_t$ in $G$ and $\cL_1 = \cL_2 = V(G)$. Iteratively remove $v \in \cL_1$ if $v$ is the center of fewer than $C n^t p^t$ copies of $S_t\in \cH$, and remove $u \in \cL_2$ if $u$ is the leaf of fewer than $C n^t p^t$ copies of $S_t\in \cH$; upon removal of these vertices, remove all stars in $\cH$ which have $v$ as its center or $u$ as a leaf. Repeat until this process terminates. We remove at most $2n \cdot C n^t p^t$ copies of $S_t$ from $\cH$, so $\cH$ has size at least $C n^{t+1} p^t$ at the end of this process, and $\cH$, $\cL_1$, $\cL_2$ have the desired properties.
	\end{proof}
	
	\begin{lemma}\label{lem:starupper}
		Let $b \geq ta$ be given, and let $p = n^{-a/b}$. Then
		\[ 
		\ex(n, S_t, T_{S_t}(a,b)^\ell \cup T_{S_{t-1}}(a,b)^\ell \cup \cdots \cup T_{S_1}(a,b)^\ell) = O_\ell(n^{t+1} p^t) .
		\]
	\end{lemma}
	
	\begin{proof}
		For ease of notation, let $F$ be a copy of $T_{S_t}(a,b)$, and let $R \subseteq V(F)$ be the root set of $F$. Let $G$ be a $T_{S_i}(a,b)^\ell$-free graph for every $i \in [t-1]$, and assume to the contrary that $G$ contains at least $3C n^{t+1} p^t$ copies of $S_t$, where $C$ is sufficiently large based on $F$. Our goal is to show that $G$ has more than $\ell$ times as many copies of $F$ as it does copies of $R$, and hence contains some graph from $F^\ell$ by the pigeonhole principle. Since $F[R]$ consists of a disjoint union of copies of $S_{t-1}$ and isolated vertices, we have an upper bound on the number of copies of $F[R]$ in $G$ given by induction. To obtain a lower bound on the number of copies of $F$ in $G$, we use Lemma~\ref{lem::star subgraph}.
		
		Let $\cH$, $\cL_1$, and $\cL_2$ be as in Lemma~\ref{lem::star subgraph} applied to $G$. We construct $F$ from `left to right,' keeping track of the number of choices for each stage of the construction. Recall that the non-root vertices of $F$, which we call the spine, is a path $u_1, \dots, u_a$; the copies of these vertices in $G$ will be called $v_1, \dots, v_a$. We begin with a copy of $S_t$ in $\cH$, of which there are at least $C n^{t+1} p^t$ choices, and we specify one of the leaves as $v_1$. Since $v_1 \in \cL_2$, there are at least $C n^t p^t$ copies of $S_t$ in $\cH$ which have $v_1$ as a leaf. 
		
		We claim we can find $d(u_1)-1$ copies of $S_t\in \mathcal{H}$ containing $v_1$ as a leaf, which are pairwise disjoint aside from $v_1$. Indeed, choosing one such star at a time, if we have chosen $k<d(u_1)-1$ stars, we have at least $Cn^{t}p^t-k$ possible choices for the next star, and at most $|V(F)|^2n^{t-1}$ total stars in $G$ contain two or more vertices in previously chosen stars. This gives us at least
		\[
		Cn^tp^t-k-|V(F)|^2n^{t-1}
		\]
		choices for the next star. Since $p=n^{-a/b}$ and $b\geq ta$, $C$ can be chosen large enough such that $Cn^tp^t-k-|V(F)|^2n^{t-1}\geq \frac{C}{2}n^tp^t$. The total number of choices of how to build this part of $F$ is then at least
		\[
		C n^{t+1} p^t (Cn^t p^t/2)^{d(u_1)-1}.
		\]
		
		To continue building $F$, choose one of the stars with $v_1$ as a leaf and call its center $v_2$ and another of its leaves $v_3$. Since $v_2 \in \cL_1$, we have that $v_2$ is the center of at least $C n^t p^t$ copies of $S_t$ in $\cH$. In particular, the degree of $v_2$ in $G$ is at least $C^{1/t} np > 2|V(F)|$, and hence we can find $d(u_2) - t$ edges incident to $v_2$ but otherwise disjoint from everything else we have found so far. Moreover, the number of choices of these edges is at least $(C^{1/t}np/2)^{d(u_2)-t}$.
		
		We proceed in the same way down the rest of the spine of $F$; when $i$ is odd, we use the fact that $v_i\in \mathcal{L}_2$ to find many stars with $v_i$ as a leaf, and when $i$ is even, we use the fact that $v_i\in \mathcal{L}_1$ to find many edges incident with $v_i$. Critically, we must have that for even $i$, $d(u_i) \geq t$, which is guaranteed by Proposition~\ref{prop::balanced:BC} and the assumption that $b \geq ta$. Observe that aside from the initial copy of $S_t$, each edge of $F$ that we constructed contributed a factor of at least $C^{1/t} n p / 2$ to our count of the number of ways to build $F$; the initial copy of $S_t$ contributed an additional factor of $n$. Since $F$ is a tree, the number of ways to build $F$ is at least
		\[ 
		n (C^{1/t} np/2)^{|E(F)|} = (C^{1/t}/2)^{|E(F)|} n^{|V(F)|} p^{|E(F)|}.
		\]
		This slightly overcounts the number of copies of $F$, as one copy of $F$ might be built in more than one way. However, one copy of $F$ could have come from at most $|V(F)|^{|V(F)|}$ different ways of building $F$ described above, so we get that the number of copies of $F$ is at least 
		\[
		\frac{(C^{1/t}/2)^{|E(F)|}}{|V(F)|^{|V(F)|}} n^{|V(F)|} p^{|E(F)|}.
		\]
		
		By induction, since $G$ is $T_{S_i}(a,b)^\ell$-free graph for every $i \in [t-1]$, the number of copies of $S_{t-1}$ in $G$ is at most $O_\ell(n^t p^{t-1})$. Since $F[R]$ consists of vertex disjoint copies of $S_{t-1}$ and isolated vertices, the number of copies of $F[R]$ in $G$ is at most an $O_\ell(n^t p^{t-1})$ factor for each copy of $S_{t-1}$ in $F[R]$, and an $n$ factor for each isolated vertex in $F[R]$. In total, the number of copies of $F[R]$ in $G$ is at most
		\[ 
		O_\ell(n^{|R|} p^{E(F[R])}).
		\]
		Let $C'=C'(\ell)$ be a constant such that the number of copies of $F[R]$ in $G$ is at most $C'n^{|R|}p^{|E(R)|}$. Then by the pigeonhole principle, we have that there exists a root set $R^*\subseteq V(G)$ such that the number of copies of $F$ with root set $R^*$ is at least 
		\[
		\left\lfloor\frac{(C^{1/t}/2)^{|E(F)|}}{|V(F)|^{|V(F)|}} \frac{n^{|V(F)|} p^{|E(F)|}}{C'n^{|R|}p^{|E(F[R])|}}\right\rfloor=\left\lfloor\frac{(C^{1/t}/2)^{|E(F)|}}{C'|V(F)|^{|V(F)|}} n^{|V(F)|-|R|} p^{|E(F)|-|E(F[R])|}\right\rfloor.
		\]
		
		Recall from Construction~\ref{construction stars} that $|V(F)|-|R|=a$, while $|E(F)|-|E(F[R])|=b$, and recall that $p=n^{-a/b}$ so the above expression becomes
		\[
		\left\lfloor\frac{(C^{1/t}/2)^{|E(F)|}}{C'|V(F)|^{|V(F)|}} n^{a}p^{b}\right\rfloor=\left\lfloor\frac{(C^{1/t}/2)^{|E(F)|}}{C'|V(F)|^{|V(F)|}}\right\rfloor.
		\]
		Taking $C$ sufficiently large compared to $F$ and $\ell$ to make the above expression at least $\ell$ yields that the number of copies of $F$ with the common root set $R^*$ is at least $\ell$, so we a copy of a graph from $F^\ell$.
	\end{proof}
	
	Theorem~\ref{thm::starsrealizable} then follows from the general lower bound in Theorem~\ref{thm::lowerbound} and the upper bound in Lemma~\ref{lem:starupper}.

	\section{Upper bound tools for \texorpdfstring{$K_t$}{2}}\label{sec::tools}
	
	In this section we give several tools which enable us to find powers of rooted $K_t$-trees in graphs with many $K_t$'s.
	
	A \emph{$t$-complex} $\cG$ is a collection of nonempty sets of size at most $t$ closed under the subset relation (aside from the empty set).\footnote{Note that a $t$-complex is just an abstract simplicial complex of dimension $t-1$. Because of this geometrically motivated, but inconvenient, $-1$ in the definition of dimension of an abstract simplicial complex, we use this nonstandard name $t$-complex.} We denote by $\cG_i$ the \emph{$i^{\text{th}}$ level of $\cG$}, the collection of all sets of $\cG$ of size $i$. The \emph{ground set} of $\cG$ is the set of all elements which appear in sets of $\cG$, denoted $V(\cG)$. That is, $V(\cG) = \{ i : \{i\} \in \cG_1\}$. The \emph{degree} of $K \in \cG$ is the number of sets of $\cG_t$ which contain $K$, denoted $\deg_\cG(K)$.
	
	For a graph $G$, the collection of (the vertex sets of) all cliques of order at most $t$ forms a $t$-complex. In Section~\ref{sec::tools:subgraph}, we show that this $t$-complex has a sub-$t$-complex taking the following form, assuming $G$ has many copies of $K_t$ and few copies of $K_i$ for $i<t$.
	
	\begin{definition}\label{def::builder:weak}
		Let $t,s,\ell\in\mathbb{N}$ with $s < t$ and $\vec{B}\in \mathbb{N}^s$. We say that a $t$-complex $\cG$ is a $(t,s,\ell,\vec{B})$-weak-builder if the following holds:
		\begin{enumerate}
			\item for every $1\leq i\leq s$, for every $K \in \cG_i$, $\deg_\cG(K) \geq B_i$,
			\item for every $K \in \cG_{s+1}$, $\deg_\cG(K) \leq \ell$.
		\end{enumerate}
	\end{definition}
	
	In Section~\ref{sec::tools:subgraph}, we show that a weak-builder also has the following strengthened form, albeit with a weakening of the parameters involved. For a $t$-complex $\cG$ and $S \subseteq V(\cG)$, let $\cG \setminus S = \{ K \setminus S : K \in \cG, K\not\subseteq S \}$.
	
	\begin{definition}\label{def::builder:strong}
		Let $t,s,\ell,r\in\mathbb{N}$ with $s < t$ and $\vec{B}\in \mathbb{N}^s$. We say that a $t$-complex $\cG$ is a $(t,s,\ell,r,\vec{B})$-strong-builder if for every $S \subseteq V(\cG)$ with $|S| \leq r$, $\cG \setminus S$ is a $(t,s,\ell,\vec{B})$-weak-builder.
	\end{definition}
	
	This form allows us to build $K_t$-trees, since we need to assure that each new copy of $K_t$ intersects appropriately with the $K_t$-tree already formed.
	
	\subsection{Finding builders}\label{sec::tools:subgraph}
	
	In our first step towards finding builders in a clique complex of a graph, we prune cliques of low degree, obtaining a large sub-complex. This almost gives a weak-builder, except for the max degree assumption (item 2 of Definition~\ref{def::builder:weak}), which requires another assumption on the graph.
	
	\begin{lemma}\label{lem::subgraph}
		Let $c, p \in \mathbb{R}^+$, and let $G$ be an $n$-vertex graph with at least $tn^t p^{\binom{t}{2}}$ copies of $K_t$, and at most $n^{i} p^{\binom{i}{2}}$ copies of $K_i$ for all $i<t$. Let $\cG$ be the $t$-complex consisting of all cliques in $G$ on at most $t$ vertices. Then there exists a $t$-complex $\cG' \subseteq \cG$ in which for every $i<t$, for every $K \in \cG'_i$, we have $\deg_{\cG'}(K) \geq n^{t-i} p^{\binom{t}{2}-\binom{i}{2}}$, and $|\mathcal{G}'_t|\geq n^tp^{\binom{t}{2}}$.
	\end{lemma}
	
	\begin{proof}
		Iteratively remove $K \in \cG$ if $|K|<t$ and $\deg(K)$ is too small; upon removing $K$, remove all sets in the $t$-complex which contain $K$ as well. Let $\cG'$ be the $t$-complex at the end of this process. In total, at most
		\[ 
		\sum_{i<t} n^i p^{\binom{i}{2}} n^{t-i} p^{\binom{t}{2}-\binom{i}{2}} =(t-1)n^t p^{\binom{t}{2}}
		\]
		sets of $\cG_t$ are removed, so
		\[
		|\mathcal{G}'_t|\geq tn^t p^{\binom{t}{2}}-(t-1)n^t p^{\binom{t}{2}}=n^tp^{\binom{t}{2}}.
		\]
	\end{proof}
	
	Next we show how weak-builders are actually strong-builders with a weakening of the parameters involved.
	
	\begin{lemma}\label{lem::codegree}
		Let $\cG$ be a $(t,s,\ell,\vec{B})$-weak-builder satisfying $\displaystyle\frac{B_s-\ell}{\binom{t}{s} \ell} \geq \frac{1}{1-2^{-1/r}}$. Then $\cG$ is a $(t,s,\ell,r,\frac{1}{2}\vec{B})$-strong-builder.
	\end{lemma}
	
	\begin{proof}
		Let $K \in \cG_i$ with $i \leq s$, and let $v \in V(G) \setminus K$. Let $\eta = 1-2^{-1/r}$, so that $(1-\eta)^r = \frac{1}{2}$. If $\deg_\cG(K \cup \{v\}) \leq \eta B_i$, then
		\[ 
		\deg_{\cG \setminus \{v\}}(K) \geq \deg_\cG(K) - \deg_\cG(K \cup \{v\}) \geq (1-\eta)B_i .
		\]
		Otherwise, $K \cup \{v\}$ is contained in at least $\eta B_i/\ell$ sets in $\cG_{s+1}$, since the degree of every set in $\cG_{s+1}$ is at most $\ell$. Removing $v$ from each of these sets in $\cG_{s+1}$, we obtain at least $\eta B_i/\ell$  sets in $(\cG \setminus \{v\})_s$ containing $K$. Then, for each $K'' \in (\cG \setminus \{v\})_t$ containing $K$, there are at most $\binom{t}{s}$ sets $K' \in (\cG \setminus \{v\})_s$ with $K\subseteq K'\subseteq K''$. This gives us that
		
		\[ 
		\deg_{\cG \setminus \{v\}}(K) \geq \frac{1}{\binom{t}{s}} \sum_{K' \in (\cG \setminus \{v\})_s : K \subseteq K'} \deg_{\cG\setminus\{v\}}(K')\geq \frac{1}{\binom{t}{s}} \sum_{K' \in (\cG \setminus \{v\})_s : K \subseteq K'} (\deg_\cG(K') - \ell).
		\]
		Since there are at least $\eta B_i/\ell$ choices for $K'$, we have
		\[ 
		\deg_{\cG \setminus \{v\}}(K) \geq \frac{1}{\binom{t}{s}} \cdot \frac{\eta B_i}{\ell} \cdot (B_s - \ell) \geq B_i>(1-\eta)B_i .
		\]
		
		Thus $\cG \setminus \{v\}$ is a $(t,s,\ell,(1-\eta)\vec{B})$-weak-builder. Iterating, since $(1-\eta)^r \geq \frac{1}{2}$, we have that $\cG$ is a $(t,s,\ell,r,\frac{1}{2}\vec{B})$-strong builder.
	\end{proof}

	\subsection{Building trees}\label{sec::tools:building}
	
	We recall Defintion~\ref{def::Htree} for $H$-trees and specialize it to the case $H = K_t$, adding some additional terminology.
	
	\begin{definition}\label{def::Kttree}
		We say that $T$ is an \emph{$K_t$-tree} if there exists a sequence of subgraphs of $T$ isomorphic to $K_t$, $K^{(1)}\subseteq T, K^{(2)}\subseteq T, \dots, K^{(b)}\subseteq T$, with $V(T)=\bigcup_{i=1}^b V(K^{(i)})$, and with the property that, for every $2\leq i\leq b$, there exists $j_i<i$ such that
		\begin{enumerate}
			\item $V(K^{(i)}) \cap \left(V(K^{(1)})\cup\cdots\cup V(K^{(i-1)})\right) = V(K^{(i)}) \cap V(K^{(j_i)})$, and
			\item $V(K^{(i)}) \cap V(K^{(j_i)})$ is neither empty nor equal to $V(K^{(i)})$.
		\end{enumerate}
		The sequence $(K^{(1)},j_2,K^{(2)},\dots,j_b,K^{(b)})$ will be called a \emph{witness} for $T$. The \emph{glue size} of $(K^{(1)},j_2,K^{(2)},\dots,j_b,K^{(b)})$ is the largest value of $s$ such that there exists some $2 \leq i \leq b$ with $|V(K^{(i)}) \cap V(K^{(j_i)})| = s$. The \emph{type} of $(K^{(1)},j_2,K^{(2)},\dots,j_b,K^{(b)})$ is $\vec{d} \in \mathbb{N}^s$, where the $k^{\text{th}}$ coordinate of $\vec{d}$, $d_k$, is the number of times $|V(K^{(i)}) \cap V(K^{(j_i)})| = k$ for $2 \leq i \leq b$. We say that a sequence $(L^{(1)}, \dots, L^{(b)})$ of sets in a $t$-complex $\mathcal{G}$ is a \emph{copy} of $(K^{(1)},j_2,K^{(2)},\dots,j_b,K^{(b)})$ in $\mathcal{G}$ if there is a bijection between $V(T)$ and $L^{(1)} \cup \cdots \cup L^{(b)}$ taking $V(K^{(i)})$ to $L^{(i)}$ for every $1 \leq i \leq b$.
	\end{definition}
	
	\begin{observation}\label{obs:CT}
		The number of copies of a witness $(K^{(1)},j_2,K^{(2)},\dots,j_b,K^{(b)})$ of a $K_t$-tree $T$ in a clique $t$-complex of a graph $G$ on $|V(T)|$ vertices is finite.
	\end{observation}
	
	\begin{lemma}\label{lemma technical count}
		Let $T$ be a $K_t$-tree with a witness $(K^{(1)},j_2,K^{(2)},\dots,j_b,K^{(b)})$ of glue size $s$ and type $\vec{d}$. Then
		\[
		|V(T)|=t+\sum_{k=1}^s d_k(t-k)
		\]
		and
		\[
		|E(T)|=\binom{t}{2}+\sum_{k=1}^sd_k\left(\binom{t}{2}-\binom{k}{2}\right).
		\]
	\end{lemma}
	
	\begin{proof}
		We simply calculate
		\begin{align*}
			|V(T)|=\left|\bigcup_{i=1}^b V(K^{(i)}) \right| &=\sum_{i=1}^b|V(K^{(i)})\setminus (V(K^{(1)})\cup\cdots\cup V(K^{(i-1)}))| \\
			&=\sum_{i=1}^b |V(K^{(i)})|-\left|V(K^{(i)})\cap (V(K^{(1)})\cup\cdots\cup V(K^{(i-1)}))\right|\\
			&=t+\sum_{i=2}^b|V(K^{(i)})|-\left|V(K^{(i)})\cap V(K^{(j_i)})\right|\\
			&=t+\sum_{i=2}^bt-\left|V(K^{(i)})\cap V(K^{(j_i)})\right| \\
			&=t+\sum_{k=1}^s d_k(t-k) .
		\end{align*}
		Similarly, 
		\begin{align*}
			|E(T)|&=\sum_{i=1}^b |E(K^{(i)})|-\left|E(K^{(i)})\cap (E(K^{(1)})\cup\cdots\cup E(K^{(i-1)}))\right|\\
			&=\binom{t}2+\sum_{i=2}^b |E(K^{(i)})|-\left|E(K^{(i)})\cap E(K^{(j_i)})\right| \\
			&=\binom{t}{2}+\sum_{k=1}^s d_k\left(\binom{t}{2}-\binom{k}{2}\right) .\qedhere
		\end{align*}
	\end{proof}
	
	\begin{lemma}\label{lem::building}
		Let $T$ be an $r$-vertex $K_t$-tree with a witness $(K^{(1)},j_2,K^{(2)},\dots,j_b,K^{(b)})$ of type $\vec{d}$ and glue size at most $s$. The number copies of $(K^{(1)},j_2,K^{(2)},\dots,j_b,K^{(b)})$ in a $(t,s,\ell,r,\vec{B})$-strong-builder $\mathcal{G}$ is at least
		\[ 
		|\cG_t| \prod_{k=1}^s B_k^{d_k} 
		\]
	\end{lemma}
	
	\begin{proof}
		There are $|\cG_t|$ choices for the image of $K^{(1)}$. Then for $2\leq i\leq b$, by the strong-builder property, there are at least $B_k$ choices for the image of $K^{(i)}$, where $k = |K^{(i)}\cap (K^{(1)}\cup\dots\cup K^{(i-1)})|$. For each $1\leq k\leq s$, as $i$ ranges from $2$ to $b$, there are $d_k$ times in which we get a multiplicative factor of $B_k$, giving us the desired count.
	\end{proof}

	\subsection{Rooting a \texorpdfstring{$K_t$}{2}-tree}\label{sec::tools:rooting}

	\begin{lemma}\label{lem::rooting}
		Let $L, \ell, s, t \in \mathbb{N}$ with $s<t$, $d\in \mathbb{Q}$ with $2d+1 \geq s+t$, and let $(T,R)$ be a rooted $K_t$-tree with glue size at most $s$ and rooted density at most $d$. There exists $c_0 \in \mathbb{R}^+$ such that the following holds for all $c \geq c_0$ and $n$ sufficiently large, with $p := cn^{-1/d}$. Let $G$ be an $n$-vertex graph such that
		\begin{enumerate}
			\item $G$ has at least $tn^t p^{\binom{t}{2}}$ copies of $K_t$,
			\item $G$ has at most $n^i p^{\binom{i}{2}}$ copies of $K_i$ for all $i<t$,
			\item every copy of $K_{s+1}$ in $G$ is in at most $\ell$ copies of $K_t$ in $G$, and
			\item $G$ contains at most $n^{|R|} p^{|E(T[R])|}$ copies of $T[R]$.
		\end{enumerate}
		Then there exists a member of the $L^{\text{th}}$ power of $(T,R)$ in $G$.
	\end{lemma}
	
	\begin{proof}
		Set $r:=|V(T)|$, $B_i:= \frac{n^{t-i}p^{\binom{t}{2}-\binom{i}{2}}}{2}$ for $1\leq i\leq s$, and $\vec{B}=(B_1,B_2,\dots,B_s)$. Let $\mathcal{G}$ be the $t$-complex whose sets are the cliques of $G$ of order at most $t$. Apply Lemma~\ref{lem::subgraph} to $\mathcal{G}$, and let $\mathcal{G}'$ be the resulting subcomplex of $\mathcal{G}$ with
		\begin{equation}\label{eq::G'}
			|\mathcal{G}'_t| \geq n^t p^{\binom{t}{2}} .
		\end{equation}
		Note that $\mathcal{G}'$ is a $(t,s,\ell,2\vec{B})$-weak-builder.
		
		Observe that, since $2d+1 \geq s+t$, we have
		\[ 
  2B_s=n^{t-s}p^{\binom{t}{2}-\binom{s}{2}}=c^{\binom{t}{2}-\binom{s}{2}} n^{\frac{(t-s)(2d+1-t-s)}{2d}} \geq c^{\binom{t}{2} - \binom{s}{2}} ,
  \]
		and hence
		\[ 
  \frac{2B_s-\ell}{\binom{t}{s}-\ell} \geq \frac{1}{1-2^{-1/r}} 
  \]
		for $c$ sufficiently large. So by Lemma~\ref{lem::codegree}, $\mathcal{G}'$ is a $(t,s,\ell,r,\vec{B})$-strong-builder. Let $(K^{(1)},j_2,\allowbreak K^{(2)},\allowbreak\dots,j_b,K^{(b)})$ be a witness of $T$, and let $\vec{d}$ be the type of this witness. By Lemma~\ref{lem::building}, $\mathcal{G}'$ contains at least $\displaystyle|\cG_t'| \prod_{k=1}^s B_k^{d_k}$ copies $(K^{(1)},j_2,K^{(2)},\dots,j_b,K^{(b)})$. By Observation~\ref{obs:CT}, the number of copies of this witness of $T$ on a particular $r$ vertices is a constant $C_T$ depending only on $T$ and this witness, and hence we have that the number of distinct copies of $T$ in $G$ must be at least
		\begin{align*}
			\frac{|\cG_t'| \prod_{k=1}^s B_k^{d_k}}{C_T}\geq\frac{n^tp^{\binom{t}{2}}\prod_{k=1}^s \left(n^{t-k}p^{\binom{t}{2}-\binom{k}{2}}/2\right)^{d_k}}{C_T}
		\end{align*}
		Let $D:=\sum_{k=1}^sd_k$. If we calculate the power of $n$ (without $p$) in the above expression, we get $t+\sum_{k=1}^s d_k(t-k)$, and the power of $p$ is $\binom{t}{2}+\sum_{k=1}^sd_k\left(\binom{t}{2}-\binom{k}{2}\right)$. Then using Lemma~\ref{lemma technical count}, we can write
		\begin{align*}
			\frac{n^tp^{\binom{t}{2}}\prod_{k=1}^s \left(n^{t-k}p^{\binom{t}{2}-\binom{k}{2}}/2\right)^{d_k}}{C_T}=\frac{n^{|V(T)|}p^{|E(T)|}}{2^D C_T}
		\end{align*}
		
		Since there are at most $n^{|R|}p^{|E(T[R])|}$ copies of $T[R]$ and at most $|R|!$ automorphisms of $T[R]$, there must be a collection of at least
		\begin{equation}\label{eq::figureoutL}
			\frac{n^{|V(T)|-|R|}p^{|E(T)|-|E(T[R])|}}{2^D C_T |R|!}
		\end{equation}
		copies of $T$ in $G$ which intersect in the same root set $R$, and in such a way that there is an isomorphism between any two copies that fix $R$. For $c$ sufficiently large, since $(T,R)$ has rooted density at most $d$, we have that \eqref{eq::figureoutL} is
		\[ \frac{c^{|E(T)|-|E(T[R])|} n^{|V(T)|-|R| - (|E(T)|-|E(T[R])|)/d}}{2^D C_T |R|!} \geq \frac{c^{|E(T)|-|E(T[R])|}}{2^D C_T |R|!} \geq L ,\]
		and hence $G$ contains a member of the $L^{\text{th}}$ power of $(T,R)$.
	\end{proof}

	\section{Constructions for cliques}\label{section constructions}
	
	We now turn to constructing the families of forbidden graphs which will achieve the desired exponents. Similar to the approach of Bukh and Conlon, we forbid large families based upon specially-chosen, tree-like constructions. Recall that Theorem~\ref{thm::lowerbound} indicates that if our aim is to show that $r$ is a realizable exponent for $K_t$, then we should use balanced rooted graphs of rooted densities at least $d$, where $r = t- \binom{t}{2}/d$. Thus, to show that all exponents in the interval $(1, t)$ are realizable for $K_t$, we must construct balanced rooted graphs with densities in $(t/2, \infty)$. Moreover, these rooted graphs should be $K_t$-trees, so that we can apply the tools of Section~\ref{sec::tools} to build them at appropriate $K_t$-densities.
	
	We begin by introducing a pair of $K_3$-tree constructions, which achieve the desired densities for $K_3$. We then generalize these constructions to create appropriate $K_t$-trees for all $t$. For this reason, while we state propositions which establish the relevant properties of these constructions for $K_3$, we postpone proofs until the constructions are described in full generality. Throughout this section, because all constructions we consider are $K_t$-trees for some $t$, we shall describe $K_t$-trees with the notation $T_{t}$ instead of $T_{K_t}$. The notation $T_2(a,b)$ always refers to the Bukh-Conlon tree on $a$ unrooted vertices and $b$ edges; refer to Construction~\ref{constr::BC} for the definition of $T_2(a,b)$.

	\subsection{\texorpdfstring{$K_3$}{2}-tree constructions}\label{sec::triangle constructions}

	We require two different constructions, depending upon the rooted density which we wish to achieve. 
	
	
	\begin{construction}\label{constr::3.1}
		Let $a,b$ be positive integers with $a+1 \leq b$. We form $T_3(a,b)$ by adding $b$ new vertices to $T_2(a,b)$, one for each edge of $T_2(a,b)$, and joining them to the endpoints of the corresponding edge. We call the vertices of the original $T_2(a,b)$ the \emph{tree vertices}, and we call the $b$ new vertices the \emph{spike vertices}. All spike vertices are unrooted, and the tree vertices retain their rooted/unrooted status.
	\end{construction}

	\begin{figure}[ht]
		\begin{center}
			\begin{tikzpicture}

				\filldraw (0,0) circle (0.05 cm);
				\filldraw (1,0) circle (0.05 cm);
				\filldraw (2,0) circle (0.05 cm);
				\filldraw (3,0) circle (0.05 cm);
				\filldraw (4,0) circle (0.05 cm);
				\filldraw (0,-1) circle (0.05 cm);
				\filldraw (2,-1) circle (0.05 cm);
				\filldraw (4,-1) circle (0.05 cm);
				\draw (0,0) -- (4,0);
				\draw (0,0) -- (0,-1);
				\draw (2,0) -- (2,-1);
				\draw (4,0) -- (4,-1);
				\draw (-1,-.7) -- (11,-.7);
				
				\filldraw (6,0) circle (0.05 cm);
				\filldraw (7,0) circle (0.05 cm);
				\filldraw (8,0) circle (0.05 cm);
				\filldraw (9,0) circle (0.05 cm);
				\filldraw (10,0) circle (0.05 cm);
				\filldraw (6,-1) circle (0.05 cm);
				\filldraw (8,-1) circle (0.05 cm);
				\filldraw (10,-1) circle (0.05 cm);
				\draw (6,0) -- (10,0);
				\draw (6,0) -- (6,-1);
				\draw (8,0) -- (8,-1);
				\draw (10,0) -- (10,-1);
				\filldraw (6.5,-0.5) circle (0.05 cm);
				\filldraw (8.5,-0.5) circle (0.05 cm);
				\filldraw (10.5,-0.5) circle (0.05 cm);
				\filldraw (6.5, 0.5) circle (0.05 cm);
				\filldraw (7.5, 0.5) circle (0.05 cm);
				\filldraw (8.5, 0.5) circle (0.05 cm);
				\filldraw (9.5, 0.5) circle (0.05 cm);
				\draw (6,0) -- (6.5, -0.5) -- (6, -1);
				\draw (8,0) -- (8.5, -0.5) -- (8, -1);
				\draw (10,0) -- (10.5, -0.5) -- (10, -1);
				\draw (6,0) -- (6.5, 0.5) -- (7,0);
				\draw (7,0) -- (7.5, 0.5) -- (8,0);
				\draw (8,0) -- (8.5, 0.5) -- (9,0);
				\draw (9,0) -- (9.5, 0.5) -- (10,0);
			\end{tikzpicture}
			
			\caption{$T_2(5,7)$ (left) and $T_3(5,7)$ (right). Rooted vertices are below the line.}    
		\end{center}
	\end{figure}
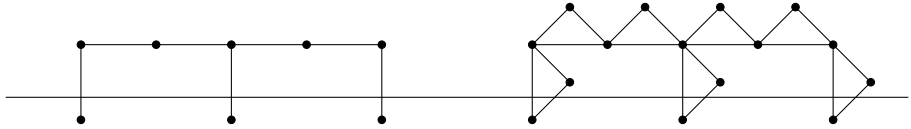
	
	\begin{prop}\label{prop::3.1}
		$T_3(a,b)$ with $a+1 \leq b \leq 2a$ is a balanced $K_3$-tree with density $\frac{3b}{a+b}$.
	\end{prop}

        We note that, since $T_3(a,b)$ contains unrooted vertices of degree $2$, it cannot be balanced if $b>2a$. For other densities, we use the following construction. First, we need the following observation about $T_2(a,b)$.

        \begin{observation}\label{obs:degrees}
            Let $a,b$ be positive integers with $2a+2 \leq b \leq 3a+1$. Label the $a$ unrooted vertices of $T_2(a,b)$ as $x_1, \dots x_a$. Then for $1 < i < a$, $x_i$ has exactly $1$ or $2$ rooted neighbors, and $x_1$ and $x_a$ each have exactly $2$ or $3$ rooted neighbors.
        \end{observation}

	\begin{construction}\label{constr::3.2}
		Let $a,b$ be positive integers such that $a$ is even and $2a + 2 \leq b$. If $b\leq 3a+1$, we form $T_3(a,b)$ by identifying root vertices in $T_2(a,b)$ and adding edges to the new root set, as follows. 
		
		Label the $a$ unrooted vertices of $T_2(a,b)$ as $x_1, \dots, x_a$. By Observation~\ref{obs:degrees}, each $x_i$ with $1 < i < a$ has either $1$ or $2$ rooted neighbors. If $x_i$ has $1$ rooted neighbor, then we label the rooted neighbor as $y_i$; if $x_i$ has $2$ rooted neighbors, then we label them as $\ell_i$ and $r_i$. By Observation~\ref{obs:degrees}, $x_1$ and $x_a$ each either have $2$ or $3$ rooted neighbors. If $x_i$ for $i \in \{1,a\}$ has $2$ rooted neighbors, then we label them as $\ell_i$ and $r_i$ as before. If $x_i$ for $i \in \{1,a\}$ has $3$ rooted neighbors, then we label two of them as $z_i$ and $w_i$ and the third as $r_1$ for $i=1$ and $\ell_a$ for $i=a$. Let $i_1 < \cdots < i_t$ be the indices where $x_i$ has at least 2 rooted neighbors.
		
		To form $T_3(a,b)$ with $2a+2 \leq b \leq 3a+1$, we perform the following identifications and additions. First, if $x_i$ for $i \in \{1,a\}$ has $3$ rooted neighbors, then we add the edge $z_iw_i$; otherwise, we add the edge $\ell_i r_i$. Second, we identify rooted vertices: For $1 \leq j < t$, we identify $r_{i_j}, \ell_{i_{j+1}}$, and all vertices $y_i$ with $i_j < i < i_{j+1}$. All vertices retain their rooted/unrooted status from $T_2(a,b)$, even after identification.

		Finally, for $b \geq 3a + 2$, we define $T_3(a,b)$ recursively to be the $K_3$-tree obtained by taking $T_3(a, b-a)$, adding $a/2$ new rooted vertices $v_1,v_2,\dots,v_{a/2}$, and then adding the edges $x_{2i-1}v_i$ and $v_ix_{2i}$ for each $1\leq i\leq a/2$. Note that $x_{2i-1},x_{2i}$ and $v_i$ induce a triangle in $T_3(a,b)$. The root set of $T_3(a,b)$ consists of all the rooted vertices of $T_3(a,b-a)$, along with all the new rooted vertices $v_1, \dots, v_{a/2}$.
	\end{construction}

	\begin{figure}[ht]
		\begin{center}
			\begin{tikzpicture}

				\filldraw (-1,0) circle (0.05 cm) node[above]{$x_1$};
				
				\filldraw (0,0) circle (0.05 cm) node[above]{$x_2$};
				
				\filldraw (1,0) circle (0.05 cm) node[above]{$x_3$};
				
				\filldraw (2,0) circle (0.05 cm) node[above]{$x_4$};
				
				\filldraw (3,0) circle (0.05 cm) node[above]{$x_5$};
				
				\filldraw (4,0) circle (0.05 cm) node[above]{$x_6$};
				
				\filldraw (-1.25,-1) circle (0.05 cm) node[below]{$\ell_1$};
				
				\filldraw (-0.75,-1) circle (0.05 cm);
				
				\draw (-0.75, -1.35) node{$r_1$};
				
				\filldraw (0,-1) circle (0.05 cm);
				
				\draw (0,-1.35) node{$y_2$};
				
				\filldraw (0.75,-1) circle (0.05 cm) node[below]{$\ell_3$};
				
				\filldraw (1.25,-1) circle (0.05 cm);
				
				\draw (1.25, -1.35) node{$r_3$};
				
				\filldraw (2,-1) circle (0.05 cm);
				\draw (2,-1.35) node{$y_4$};
				
				\filldraw (2.75,-1) circle (0.05 cm) node[below]{$\ell_5$};
				
				\filldraw (3.25,-1) circle (0.05 cm);
				
				\draw (3.25, -1.35) node{$r_5$};
				
				\filldraw (3.75,-1) circle (0.05 cm) node[below]{$\ell_6$};
				
				\filldraw (4.25,-1) circle (0.05 cm);
				
				\draw (4.25, -1.35) node{$r_6$};
				
				\draw (-1,0) -- (4,0);
				\draw (-1,0) -- (-1.25,-1);
				\draw (-1,0) -- (-0.75,-1);
				\draw (0,0) -- (0,-1);
				\draw (1,0) -- (0.75,-1);
				\draw (1,0) -- (1.25,-1);
				\draw (2,0) -- (2,-1);
				\draw (3,0) -- (2.75,-1);
				\draw (3,0) -- (3.25,-1);
				\draw (4,0) -- (3.75,-1);
				\draw (4,0) -- (4.25,-1);

				\filldraw (6,0) circle (0.05 cm) node[above]{$x_1$};
				\filldraw (7,0) circle (0.05 cm) node[above]{$x_2$};
				\filldraw (8,0) circle (0.05 cm) node[above]{$x_3$};
				\filldraw (9,0) circle (0.05 cm) node[above]{$x_4$};
				\filldraw (10,0) circle (0.05 cm) node[above]{$x_5$};
				\filldraw (11,0) circle (0.05 cm) node[above]{$x_6$};
				\filldraw (6,-1) circle (0.05 cm);
				\filldraw (7, -1) circle (0.05 cm);
				\filldraw (9,-1) circle (0.05 cm);
				\filldraw (10.5,-1) circle (0.05 cm);
				\filldraw (11,-1) circle (0.05 cm);
				
				\draw (6,0) -- (11,0);
				
				\draw (6,0) -- (6,-1);
				
				\draw (6,0) -- (7,-1);
				\draw (7,0) -- (7,-1);
				\draw (8,0) -- (7,-1);
				
				\draw (6,-1) -- (7,-1);
				
				\draw (8,0) -- (9,-1);
				\draw (9,0) -- (9,-1);
				\draw (10,0) -- (9,-1);
				
				\draw (10,0) -- (10.5,-1);
				\draw (11,0) -- (10.5,-1);
				
				\draw (11,0) -- (11,-1);
				\draw (10.5,-1) -- (11,-1);
				
				\draw (-2,-.7) -- (12,-.7);

			\end{tikzpicture}
			
			\caption{$T_2(6,15)$ (left) and $T_3(6,15)$ (right). Rooted vertices are below the line.}    
		\end{center}
	\end{figure}
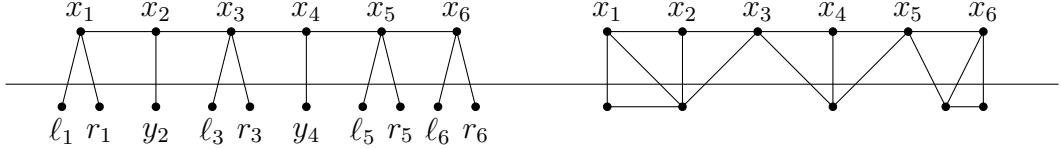

	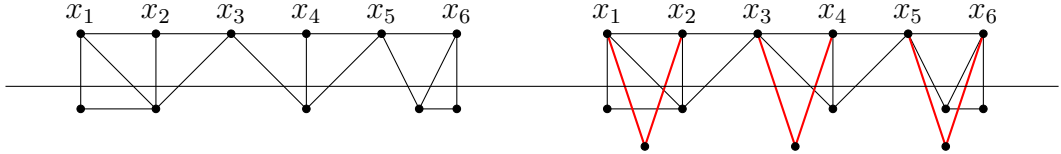
\begin{figure}[ht]
		\begin{center}
			\begin{tikzpicture}

				\filldraw (-1,0) circle (0.05 cm) node[above]{$x_1$};
				\filldraw (0,0) circle (0.05 cm) node[above]{$x_2$};
				\filldraw (1,0) circle (0.05 cm) node[above]{$x_3$};
				\filldraw (2,0) circle (0.05 cm) node[above]{$x_4$};
				\filldraw (3,0) circle (0.05 cm) node[above]{$x_5$};
				\filldraw (4,0) circle (0.05 cm) node[above]{$x_6$};
				\filldraw (-1,-1) circle (0.05 cm);
				\filldraw (0, -1) circle (0.05 cm);
				\filldraw (2,-1) circle (0.05 cm);
				\filldraw (3.5,-1) circle (0.05 cm);
				\filldraw (4,-1) circle (0.05 cm);
				
				\draw (-1,0) -- (4,0);
				
				\draw (-1,0) -- (-1,-1);
				
				\draw (-1,0) -- (0,-1);
				\draw (0,0) -- (0,-1);
				\draw (1,0) -- (0,-1);
				
				\draw (-1,-1) -- (0,-1);
				
				\draw (1,0) -- (2,-1);
				\draw (2,0) -- (2,-1);
				\draw (3,0) -- (2,-1);
				
				\draw (3,0) -- (3.5,-1);
				\draw (4,0) -- (3.5,-1);
				
				\draw (4,0) -- (4,-1);
				\draw (3.5,-1) -- (4,-1);

				\filldraw (6,-1) circle (0.05 cm);
				\filldraw (7, -1) circle (0.05 cm);
				\filldraw (9,-1) circle (0.05 cm);
				\filldraw (10.5,-1) circle (0.05 cm);
				\filldraw (11,-1) circle (0.05 cm);
				
				\draw (6,0) -- (11,0);
				
				\draw (6,0) -- (6,-1);
				
				\draw (6,0) -- (7,-1);
				\draw (7,0) -- (7,-1);
				\draw (8,0) -- (7,-1);
				
				\draw (6,-1) -- (7,-1);
				
				\draw (8,0) -- (9,-1);
				\draw (9,0) -- (9,-1);
				\draw (10,0) -- (9,-1);
				
				\draw (10,0) -- (10.5,-1);
				\draw (11,0) -- (10.5,-1);
				
				\draw (11,0) -- (11,-1);
				\draw (10.5,-1) -- (11,-1);
				
				\draw[thick, color = red] (6,0) -- (6.5, -1.5) -- (7,0) ; 
				
				\draw[thick, color = red] (8,0) -- (8.5, -1.5) -- (9,0) ;
				
				\draw[thick, color = red] (10,0) -- (10.5, -1.5) -- (11,0) ;
				
				\filldraw (6.5, -1.5) circle (0.05 cm);
				
				\filldraw (8.5, -1.5) circle (0.05 cm);
				
				\filldraw (10.5, -1.5) circle (0.05 cm);
				
				\filldraw (6,0) circle (0.05 cm) node[above]{$x_1$};
				\filldraw (7,0) circle (0.05 cm) node[above]{$x_2$};
				\filldraw (8,0) circle (0.05 cm) node[above]{$x_3$};
				\filldraw (9,0) circle (0.05 cm) node[above]{$x_4$};
				\filldraw (10,0) circle (0.05 cm) node[above]{$x_5$};
				\filldraw (11,0) circle (0.05 cm) node[above]{$x_6$};
				
				\draw (-2,-.7) -- (12,-.7);

			\end{tikzpicture}
			
			\caption{$T_3(6,21)$ (right) is obtained from $T_3(6,15)$ (left) by the addition of disjoint rooted triangles. Rooted vertices are below the line.}    
		\end{center}
	\end{figure}
	
	\begin{prop}\label{prop::3.2}
		$T_3(a,b)$ with $b \geq 2a+2$ and $a$ even is a balanced $K_3$-tree with density $b/a$. Moreover, the subgraph of $T_3(a,b)$ induced by its root set is a disjoint union of isolated vertices and edges.
	\end{prop}

        While Construction~\ref{constr::3.2} could extend to $T_3(a,2a+1)$, we exclude this case in Proposition~\ref{prop::3.2} because there is a path with two edges induced by the root set in $T_3(a,2a+1)$.
	
	\subsection{\texorpdfstring{$K_t$}{2}-tree constructions}
	
	We now describe generalizations of the constructions in Subsection~\ref{sec::triangle constructions} which we shall use to obtain the desired range of rooted densities for all cliques. Firstly, we generalize Construction~\ref{constr::3.1}.
	
	\begin{construction}\label{generalized spike construction}
		
		Fix positive integers $a,b,s,t$, with $t \geq 3$, $s \leq \frac{t}{2}$, and $a+1 \leq b$. To construct a \textbf{Type 1 $K_t$-tree} with parameters $a,b,s$, we replace each vertex $v_i$ of $T_2(a,b)$ with a set $S_i$ of $s$ vertices, and replace each edge $v_iv_j$ of $T_2(a,b)$ with a set $S_{i,j}$ of $t - 2s$ spike vertices. The sets $S_i$ and $S_{i,j}$ each induce cliques. For $1 \leq i < j \leq a$, the vertices of $S_i$ are adjacent to the vertices of $S_j$ if and only if $v_iv_j$ is an edge of $T_2(a,b)$. For each edge $v_iv_j$ of $T_2(a,b)$, the vertices in $S_{i,j}$ are adjacent to all vertices of $S_i$ and $S_j$.
		
		We denote the above construction as $T_t^1(a,b,s)$. The root of $T_t^1(a,b,s)$ consists of the union of the sets $S_i$ such that $v_i$ was rooted in $T_2(a,b)$.   
	\end{construction}

	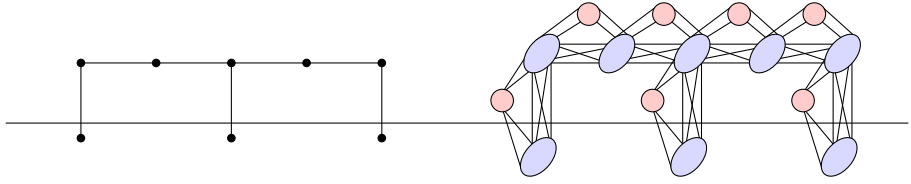
\begin{figure}[ht]
		\begin{center}
			\begin{tikzpicture}

				\filldraw (0,0) circle (0.05 cm);
				\filldraw (1,0) circle (0.05 cm);
				\filldraw (2,0) circle (0.05 cm);
				\filldraw (3,0) circle (0.05 cm);
				\filldraw (4,0) circle (0.05 cm);
				\filldraw (0,-1) circle (0.05 cm);
				\filldraw (2,-1) circle (0.05 cm);
				\filldraw (4,-1) circle (0.05 cm);
				\draw (0,0) -- (4,0);
				\draw (0,0) -- (0,-1);
				\draw (2,0) -- (2,-1);
				\draw (4,0) -- (4,-1);

				\filldraw (6,0) circle (0.05 cm);
				\filldraw (7,0) circle (0.05 cm);
				\filldraw (8,0) circle (0.05 cm);
				\filldraw (9,0) circle (0.05 cm);
				\filldraw (10,0) circle (0.05 cm);

				\filldraw (6.25,0.25) circle (0.05 cm);
				\filldraw (7.25,0.25) circle (0.05 cm);
				\filldraw (8.25,0.25) circle (0.05 cm);
				\filldraw (9.25,0.25) circle (0.05 cm);
				\filldraw (10.25,0.25) circle (0.05 cm);

				\draw (6,0) -- (10,0);
				\draw (6,0) -- (6,-1.325);
				\draw (8,0) -- (8,-1.325);
				\draw (10,0) -- (10,-1.325);
				
				\draw (6.25,0.25) -- (10.25, 0.25);
				
				\draw (6.25,0.25) -- (6.25,-1.125);
				\draw (8.25,0.25) -- (8.25,-1.125);
				\draw (10.25,0.25) -- (10.25,-1.125);

				\draw (6,0) -- (7.25, 0.25);
				\draw (7,0) -- (8.25, 0.25);
				\draw (8,0) -- (9.25, 0.25);
				\draw (9,0) -- (10.25, 0.25);
				\draw (7,0) -- (6.25, 0.25);
				\draw (8,0) -- (7.25, 0.25);
				\draw (9,0) -- (8.25, 0.25);
				\draw (10,0) -- (9.25, 0.25);
				
				\draw (6,0) -- (6.25, -1.125);
				\draw (8,0) -- (8.25, -1.125);
				\draw (10,0) -- (10.25, -1.125);
				\draw (6.25, 0.25) -- (6,-1.375);
				\draw (8.25, 0.25) -- (8,-1.375);
				\draw (10.25, 0.25) -- (10,-1.375);
				
				\draw (5.95,0.14) -- (5.6, -0.4);
				\draw (6.1,0) -- (5.6, -0.4);
				\draw (6.1,-1.2) -- (5.6, -0.6);
				\draw (5.85,-1.4) -- (5.6, -0.6);

				\draw (7.95,0.14) -- (7.6, -0.4);
				\draw (8.1,0) -- (7.6, -0.4);
				\draw (8.1,-1.2) -- (7.6, -0.6);
				\draw (7.85,-1.4) -- (7.6, -0.6);
				
				\draw (9.95,0.14) -- (9.6, -0.4);
				\draw (10.1,0) -- (9.6, -0.4);
				\draw (10.1,-1.2) -- (9.6, -0.6);
				\draw (9.85,-1.4) -- (9.6, -0.6);

				\draw (6.65,0.75) -- (6.05,0.3);
				\draw (6.65,0.55) -- (6.25,0.3);
				\draw (6.85, 0.75) -- (7.35,0.3);
				\draw (6.85, 0.55) -- (7.15,0.3);

				\draw (7.65,0.75) -- (7.05,0.3);
				\draw (7.65,0.55) -- (7.25,0.3);
				\draw (7.85, 0.75) -- (8.35,0.3);
				\draw (7.85, 0.55) -- (8.15,0.3);
				
				\draw (8.65,0.75) -- (8.05,0.3);
				\draw (8.65,0.55) -- (8.25,0.3);
				\draw (8.85, 0.75) -- (9.35,0.3);
				\draw (8.85, 0.55) -- (9.15,0.3);
				
				\draw (9.65,0.75) -- (9.05,0.3);
				\draw (9.65,0.55) -- (9.25,0.3);
				\draw (9.85, 0.75) -- (10.35,0.3);
				\draw (9.85, 0.55) -- (10.15,0.3);
				
				\filldraw[rotate around = {-40.25:(6.125,0.125)}, fill = blue!15!white] (6.125,0.125) ellipse (0.18cm and 0.3 cm);
				
				\filldraw[rotate around = {-40.25:(7.125,0.125)}, fill = blue!15!white] (7.125,0.125) ellipse (0.18cm and 0.3 cm);
				
				\filldraw[rotate around = {-40.25:(8.125,0.125)}, fill = blue!15!white] (8.125,0.125) ellipse (0.18cm and 0.3 cm);
				
				\filldraw[rotate around = {-40.25:(9.125,0.125)}, fill = blue!15!white] (9.125,0.125) ellipse (0.18cm and 0.3 cm);
				
				\filldraw[rotate around = {-40.25:(10.125,0.125)}, fill = blue!15!white] (10.125,0.125) ellipse (0.18cm and 0.3 cm);

				\filldraw[rotate around = {-40.25:(6.08,-1.25)}, fill = blue!15!white] (6.08,-1.25) ellipse (0.18cm and 0.3 cm);

				\filldraw[rotate around = {-40.25:(8.08,-1.25)}, fill = blue!15!white] (8.08,-1.25) ellipse (0.18cm and 0.3 cm);

				\filldraw[rotate around = {-40.25:(10.08,-1.25)}, fill = blue!15!white] (10.08,-1.25) ellipse (0.18cm and 0.3 cm);
				
				\draw (-1, -0.8) -- (11, -0.8);

				\filldraw[fill = red!20!white] (5.6, -0.5) circle (0.15 cm);
				
				\filldraw[fill = red!20!white] (7.6, -0.5) circle (0.15 cm);
				
				\filldraw[fill = red!20!white] (9.6, -0.5) circle (0.15 cm);
				
				\filldraw[fill = red!20!white] (6.75, 0.65) circle (0.15 cm);
				
				\filldraw[fill = red!20!white] (7.75, 0.65) circle (0.15 cm);
				
				\filldraw[fill = red!20!white] (8.75, 0.65) circle (0.15 cm);
				
				\filldraw[fill = red!20!white] (9.75, 0.65) circle (0.15 cm);
				
			\end{tikzpicture}
			
			\caption{$T_2(5,7)$ (left) and $T_{t}^1(5,7,s)$ (right). Blue shaded ellipses represent cliques of size $s$, while pink shaded circles represent cliques of size $t - 2s$. Rooted vertices are below the line.}    
		\end{center}
	\end{figure}

\begin{observation}\label{observation density of dt1}
	For all $t \geq 3$ and positive integers $a,b,s$ such that $a+1 \leq b $ and $2s \leq t$, the construction $T_t^1(a,b,s)$ has rooted density 
	\[
	d_{t}^1(a,b,s) := \frac{a\binom{s}{2} + b\left( \binom{t}{2} - 2 \binom{s}{2} \right)}{as + b(t - 2s)}.
	\]
	Moreover, $T_t^1(a,b,s)$ is a $K_t$-tree with glue size $s$, and the subgraph of $T_t^1(a,b,s)$ induced by the root set is a disjoint union of cliques of size $s$.
\end{observation}

\begin{proof}
	The value of $d_{t}^1(a,b,s)$ is immediate: each of the $a$ unrooted vertices in $T_2(a,b)$ corresponds to an $s$-vertex clique in $T_t^1(a,b,s)$, contributing a total of $sa$ unrooted vertices and $a \binom{s}{2}$ unrooted edges. Each edge $v_iv_j$ in $T_2(a,b)$ contributes all possible edges between the copies of $K_s$ corresponding to $v_i, v_j$, an additional clique on $t-2s$ vertices, and all edges between this copy of $K_{t - 2s}$ and the two copies of $K_s$ corresponding to $v_i$ and $v_j$.  Thus $T_t^1(a,b,s)$ has $as + b(t-2s)$ unrooted vertices and $a\binom{s}{2} + b\left( \binom{t}{2} - 2\binom{s}{2} \right)$ unrooted edges.
	
	To see that $T_t^1(a,b,s)$ is a $K_t$-tree, note that each edge of $T_2(a,b)$ corresponds to a copy of $K_t$ in $T_t^1(a,b,s)$, and these copies of $K_t$ intersect each other in $s$-sets. Finally, the vertices in the root of $T_t^1(a,b,s)$ correspond to vertices in the root of $T_2(a,b)$, and since there were no edges in the root of $T_2(a,b)$, the only edges in the root of $T_t^1(a,b,s)$ are the edges contained in the copies of $K_s$ corresponding to each rooted vertex of $T_2(a,b)$.
\end{proof}
	
	To show that Construction~\ref{generalized spike construction} is balanced, we require the following technical lemmas. Recall that $N[v]$ denotes the \emph{closed neighborhood of $v$}, that is the set containing $v$ and all vertices adjacent to $v$ in the given graph.
	
	\begin{lemma}\label{all or nothing lemma}
		Let $(F,R)$ be a rooted graph, and let $T\subseteq V(F)\setminus R$ be a set of vertices such that $N[u]=N[v]$ for all $u,v\in T$. Then for any set $S\subseteq V(F)\setminus R$, either
		\[
		d(S\setminus T)\leq d(S),\text{ or }d(S\cup T)\leq d(S).
		\]
	\end{lemma}
	
	\begin{proof}
		Let $S':=S\setminus T$, and let $T'\subseteq T$ be any set of $x$ vertices (for some $0\leq x\leq |T|$). For a vertex $v\in T$, let $r$ denote the number of edges from $v$ to $V(F) \setminus (S' \cup T)$. Then,
		\[
		d(S'\cup T')=\frac{e(S')+rx+\binom{x}{2}+x(|T|-x)}{|S'|+x}=\frac{-x^2/2+(r+|T|-1/2)x+e(S')}{|S'|+x}=:f(x).
		\]
		Considering $f$ as a function of the continuous variable $x$ on the interval $[0,|T|]$, we have
		\[
		f'(x)=\frac{-x^2/2-|S'|x+(r+|T|-1/2)|S'|-e(S')}{(|S'|+x)^2}.
		\]
		$f'$ has roots $x=-|S'|\pm \sqrt{|S'|^2+2(r+|T|-1/2)|S'|-2e(S')}$. If neither of these are positive reals, then $f'$ is monotone on $[0,|T|]$, and thus is minimized at either $0$ or $|T|$. Otherwise, $f'$ has exactly one positive real root $\left(x=-|S'|+\sqrt{|S'|^2+2(r+|T|-1/2)|S'|-2e(S')}\right)$, and for this to be positive, we must have $e(S')<(r+|T|-1/2)|S'|$. Using the preceding inequality, we can see that $f'(0)>0$, so $f$ is increasing at $0$ and has at most one critical value on $[0,|T|]$, so $f$ is minimized at $x=0$ or $x=|T|$.
	\end{proof}

For a graph $G$ and disjoint sets $X$ and $Y$ of vertices of $G$, the induced bipartite graph $G[X,Y]$ is the bipartite graph with parts $X$ and $Y$ which includes exactly those edges between $X$ and $Y$ in $G$.

\begin{lemma}\label{lemma density is an average}
	Let $(G,R)$ be a rooted graph, and let $X,Y\subseteq V(G)\setminus R$ be disjoint sets of unrooted vertices. Then,
	\[
	\min\left\{d(X),\frac{e(Y)-|E(G[X,Y])|}{|Y|}\right\}\leq d(X\cup Y)\leq \max\left\{d(X),\frac{e(Y)-|E(G[X,Y])|}{|Y|}\right\}
	.\]
\end{lemma}

\begin{proof}
Observing that 
\[ d(X\cup Y)=\frac{e(X)+e(Y)-|E(G[X,Y])|}{|X|+|Y|} \quad \text{ and } \quad d(X) = \frac{e(X)}{|X|} ,\]
the lemma directly follows from the so-called mediant inequality: for fractions $a/b \leq c/d$ with $b, d > 0$, we have
\[ \frac{a}{b} \leq \frac{a+c}{b+d} \leq \frac{c}{d} .\qedhere\]
\end{proof}
	
	Now, we are ready to demonstrate that Construction~\ref{generalized spike construction} is balanced in the appropriate density ranges.
	
	\begin{prop}\label{proposition generalized spike properties}
		Let $t \geq 3$ and positive integers $a,b,s$ be such that $a+1 \leq b $ and $2s \leq t$. If either $2s=t$ or $d_{t}^1(a,b,s) \leq \frac{t + 2s - 1}{2}$, then $T_t^1(a,b,s)$ is balanced.
	\end{prop}
	
	\begin{proof}
		For ease of notation, let $T_t:=T_t^1(a,b,s)$, $T_2:=T_2(a,b)$ and $d_t:=d_{t}^1(a,b,s)$. Let $A\subseteq V(T_t)\setminus R$ be a non-empty set; we will prove that $d(A)\geq d_t$. We may assume that $A$ induces a connected subgraph of $T_t$, since $d(A)$ is bounded below by the density of the sparsest connected component of $T_t[A]$. Moreover, by Lemma~\ref{all or nothing lemma}, for any $v_i \in V(T_2)$, we may assume that the corresponding vertex set $S_i \subset V(T_t)$ is either contained in $A$ or disjoint from $A$. Likewise, for any edge $v_iv_j \in E(T_2)$, we may assume that the corresponding vertex set $S_{i,j} \subset V(T_t)$ is either contained in $A$ or disjoint from $A$.
		
		\textbf{Case 1:} $A=S_{i,j}$ for some valid choice of $i,j$. If $2s=t$, then $S_{i,j}$ is empty, so we may assume $2s<t$. We then have
		\[
		d(A) = \frac{\binom{t-2s}{2} + 2s(t-2s)}{t-2s} = \frac{t + 2s - 1}{2}\geq d_t.
		\]
		
		\textbf{Case 2:} $A\cap S_i\neq \emptyset$ for some valid choice of $i$. Set
		\[
		V' := \{v_i \in V(T_2)\mid S_i \subseteq A\}
		\]
		and
		\[
		E' := \{v_iv_j \in E(T_2)\mid S_{i,j} \subseteq A\}.
		\]
		Note that, since $A$ induces a connected subgraph of $T_t$, every edge in $E'$ is incident to some vertex in $V'$. We next establish the following claim, showing that we may assume that every edge incident to $V'$ is contained in $E'$.
		
		\begin{claim}\label{spike clique claim}
			
			Suppose $v_i \in V'$ and there exists $v_iv_j \in E(T_2)\setminus E'$. Then $d(A \cup S_{i,j}) \geq d_t$ implies $d(A) \geq d_t$.
			
		\end{claim}
	
	\begin{proof}[Proof of Claim~\ref{spike clique claim}]
		By Lemma~\ref{lemma density is an average}, we have that
		\[
		d(A\cup S_{i,j})\leq \max\left\{d(A),\frac{e(S_{i,j})-|E(T_t[A,S_{i,j}])|}{|S_{i,j}|}\right\}.
		\]
		
		Note that the condition $b > a$ implies $d_t > \frac{t+s-1}{2} > \frac{t-1}{2}$. We have that $|E(T_t[A,S_{i,j}])|$ is either $2s(t-2s)$ (if both $v_i,v_j \in V'$) or $s(t-2s)$ (if $v_i \in V'$ but $v_j \not\in V'$), and $e(S_{i,j})=\binom{t-2s}{2}+2s(t-2s)$, so
		\begin{equation}
			\frac{e(S_{i,j})-|E(T_t[A,S_{i,j}])|}{|S_{i,j}|}\leq \frac{\binom{t-2s}{2}+s(t-2s)}{t-2s}=\frac{t-1}{2}<d_t.
		\end{equation}
	Thus, if $d(A \cup S_{i,j}) \geq d_t$, then we have $d(A) \geq d(A \cup S_{i,j}) \geq d_t$.
	\end{proof}
		
		By Claim~\ref{spike clique claim}, showing that $d(A \cup \bigcup \{S_{i,j}: v_i \in V'\}) \geq d_t$ proves that $d(A) \geq d_t$, so we may assume $\{S_{i,j}: v_i \in V'\} \subseteq A$, or in other words, that $E'$ is equal to the set of edges incident to $V'$ in $T_2$. Thus, in $T_2$, we have 
		\[
		d(V') = \frac{|E'|}{|V'|} \geq \frac{b}{a}=d(T_2)
		\] 
		since $T_2$ is balanced. This implies the following inequality:
		\[
		d(A) = \frac{|V'|\binom{s}{2} + |E'|\left(\binom{t}{2} - 2\binom{s}{2} \right)}{|V'|s + |E'|(t- 2s)} \geq \frac{a\binom{s}{2} + b(\binom{t}{2} - 2\binom{s}{2} )}{as + b(t- 2s)} = d_t ,
		\]
            which is easy to see if $t = 2s$, and otherwise follows from the fact that $\frac{\binom{s}{2}}{s} \leq \frac{\binom{t}{2} - 2\binom{s}{2}}{t-2s}$ (which itself is equivalent to $s \leq t$).
	\end{proof}

	Note that when $t$ is odd, $T_t^1(a,b,s)$ always has local density obstructions; the sets $S_{i,j}$ have size $t-2s\neq 0$, and $d(S_{i,j})=\frac{t + 2s - 1}{2}$, so the density condition in Proposition~\ref{proposition generalized spike properties} cannot be removed. This prevents our first construction from covering the entire range of densities necessary for odd $t$. On the other hand, for even $t$ we can construct balanced $K_t$-trees $T_t^1(a,b,s)$ of arbitrarily high rooted density, taking $s = t/2$ to avoid any local density obstruction. Thus, for our purposes, the possible densities achieved by Construction~\ref{generalized spike construction} are in $(t/2, t-1]$ when $t$ is odd and in $(t/2, \infty)$ when $t$ is even. It remains to verify that all rational densities in these intervals are achievable by Construction~\ref{generalized spike construction}; we do so now.
	
	\begin{prop}\label{prop::rationaltype1}
		Let $t\geq 3$, $d\in \mathbb{Q}$ with $d > t/2$ and $s\in \mathbb{N}$. If either of the following are satisfied:
		\begin{itemize}
			\item $s<t/2$ and $d\in \left( \frac{t+s-1}{2}, \frac{t + 2s - 1}{2} \right]$, or
			\item $t$ is even, $s=t/2$ and $d\in \left(\frac{3t-2}{4},\infty\right)$,
		\end{itemize}
		then there exists integers $0<a<b$ such that 
		\[
		d_t^1(a,b,s) = d.
		\]
\end{prop}
	
	\begin{proof}
		Let $d=\frac{x}{y}$ for $x,y\in\mathbb{N}$. For positive integers $s,t$ with $s < t/2$, we observe that for any rational $\frac{x}{y} \in \left( \frac{t+s-1}{2}, \frac{t + 2s - 1}{2} \right]$, we can achieve $d_{t}^1(a,b,s) = \frac{x}{y}$ by choosing
		\[
		a = y\left( \binom{t}{2} - 2 \binom{s}{2}\right) - x(t - 2s)
		\]
		and
		\[
		b = sx - y \binom{s}{2}.
		\]
		We must verify that for $\frac{x}{y}$ in the desired interval, $a$ and $b$ satisfy the conditions $a > 0$ and $a < b$. The condition that $a>0$ is equivalent to 
		\[
		\frac{\binom{t}{2} - 2 \binom{s}{2}}{t - 2s} > \frac{x}{y};
		\]
		note that 
		\[
		\frac{\binom{t}{2} - 2 \binom{s}{2}}{t - 2s} > \frac{t + 2s - 1}{2}.
		\]
		Thus, $\frac{x}{y} \leq \frac{t + 2s - 1}{2}$ implies $a > 0$. 
		The condition $a > b$ is equivalent to 
		\[
		\frac{x}{y} > \frac{t + s - 1}{2}.
		\]

		If $t$ is even and we take $s = \frac{t}{2}$, then for any rational $\frac{x}{y} \in \left[ \frac{3t - 2}{4}, \infty \right)$, we can achieve density $\frac{x}{y}$ by choosing $a,b$ as above. Indeed, when $s=\frac{t}{2}$, we have
		\[
		a=y\left( \binom{2s}{2} - 2 \binom{s}{2}\right)=ys^2>0,
		\]
		while the condition that $a < b$ remains equivalent to 
		\[
		\frac{x}{y} > \frac{t + s - 1}{2}=\frac{3t - 2}{4}.
		\]
		Thus, we can achieve all densities $\frac{x}{y} \in \left( \frac{3t - 2}{4}, \infty \right)$.
	\end{proof}

	When $t$ is odd, we require a generalization of Construction~\ref{constr::3.2} to achieve rational densities above $t-1$. 
	
	\begin{construction}\label{generalized fan construction}
		
		Fix positive integers $a,b,t$, such that $t \geq 3$ is odd, $a$ is even, and $2a + 2 \leq b$. To construct a \textbf{Type 2 $K_t$-tree} with parameters $a,b$, we replace each unrooted vertex of $T_3(a,b)$, and the leftmost and rightmost rooted neighbors of $x_1$ and $x_a$ respectively ($\ell_1$ or $z_1$, and $r_a$ or $w_a$, respectively, see Construction~\ref{constr::3.2}), by a clique of size $\frac{t-1}{2}$. Each edge of $T_3(a,b)$ is replaced by a complete bipartite graph between the cliques at its endpoints.
		
		We denote the above construction as $T_t^2(a,b)$. The root of $T_t^2(a,b)$ consists of all the vertices which were rooted in $T_3(a,b)$, along with the new vertices in the cliques which replaced the leftmost and rightmost rooted neighbors of $x_1$ and $x_a$ respectively.
		
	\end{construction}

	\begin{figure}[ht]
		\begin{center}
			\begin{tikzpicture}

				\filldraw (-1,0) circle (0.05 cm);
				
				\filldraw (0,0) circle (0.05 cm);
				
				\filldraw (1,0) circle (0.05 cm);
				
				\filldraw (2,0) circle (0.05 cm);
				
				\filldraw (3,0) circle (0.05 cm);
				
				\filldraw (4,0) circle (0.05 cm);
				
				\filldraw (-1,-1) circle (0.05 cm);
				\filldraw (0, -1) circle (0.05 cm);
				\filldraw (2,-1) circle (0.05 cm);
				\filldraw (3.5,-1) circle (0.05 cm);
				\filldraw (4,-1) circle (0.05 cm);
				
				\draw (-1,0) -- (4,0);
				
				\draw (-1,0) -- (0,-1);
				\draw (0,0) -- (0,-1);
				\draw (1,0) -- (0,-1);
				\draw (-1,0) -- (-1,-1);
				\draw (-1,-1) -- (0,-1);
				
				\draw (1,0) -- (2,-1);
				\draw (2,0) -- (2,-1);
				\draw (3,0) -- (2,-1);
				
				\draw (3,0) -- (3.5,-1);
				\draw (4,0) -- (3.5,-1);
				
				\draw (4,0) -- (4,-1);
				\draw (3.5,-1) -- (4,-1);


				\draw (6,0) -- (7,-1);
				\draw (7,0) -- (7,-1);
				\draw (8,0) -- (7,-1);
				
				\draw (6,-1) -- (7,-1);
				
				\draw (8,0) -- (9,-1);
				\draw (9,0) -- (9,-1);
				\draw (10,0) -- (9,-1);
				
				\draw (10,0) -- (10,-1);
				\draw (11,0) -- (10,-1);

				\draw (10,-1) -- (11,-1);
				
				\draw (6,0.2) -- (7,-0.2);
				\draw (6,-0.2) -- (7,0.2);
				\draw (6,0.2) -- (7,0.2);
				\draw (6,-0.2) -- (7,-0.2);
				
				\draw (7,0.2) -- (8,-0.2);
				\draw (7,-0.2) -- (8,0.2);
				\draw (7,0.2) -- (8,0.2);
				\draw (7,-0.2) -- (8,-0.2);
				
				\draw (8,0.2) -- (9,-0.2);
				\draw (8,-0.2) -- (9,0.2);
				\draw (8,0.2) -- (9,0.2);
				\draw (8,-0.2) -- (9,-0.2);
				
				\draw (9,0.2) -- (10,-0.2);
				\draw (9,-0.2) -- (10,0.2);
				\draw (9,0.2) -- (10,0.2);
				\draw (9,-0.2) -- (10,-0.2);
				
				\draw (10,0.2) -- (11,-0.2);
				\draw (10,-0.2) -- (11,0.2);
				\draw (10,0.2) -- (11,0.2);
				\draw (10,-0.2) -- (11,-0.2);
				
				\draw (6.3,-0.1) -- (5.8, - 0.8);
				\draw (5.75,0.1) -- (6.2, -0.9);
				\draw (5.8, 0.1) -- (5.8, -1);
				\draw (6.2, 0.1) -- (6.2, -1.1);
				
				\draw (11.25,-0.08) -- (10.8, - 0.8);
				\draw (10.75,0.1) -- (11.2, -0.9);
				\draw (10.8, 0.1) -- (10.8, -1);
				\draw (11.2, 0.1) -- (11.2, -1.1);
				
				\draw (6,-0.8) -- (7,-1);
				\draw (6,-1.2) -- (7,-1);
				
				\draw (11,-0.8) -- (10,-1);
				\draw (11,-1.2) -- (10,-1);
				
				\draw (5.8, -0.1) -- (7,-1);
				\draw (6.2, 0.1) -- (7,-1);
				
				\draw (10.8, 0.1) -- (10,-1);
				\draw (11.2, -0.1) -- (10,-1);
				
				\draw (7,-1) -- (6.8,0);
				\draw (7,-1) -- (7.2,0);
				
				\draw (9,-1) -- (8.8,0);
				\draw (9,-1) -- (9.2,0);
				
				\draw (10,-1) -- (9.8,0);
				\draw (10,-1) -- (10.2,0);
				
				\draw (9.8, 0.1) -- (9,-1);
				\draw (10.2, -0.1) -- (9,-1);
				
				\draw (7.8, 0.1) -- (7,-1);
				\draw (8.2, -0.1) -- (7,-1);
				
				\draw (7.8, -0.1) -- (9,-1);
				\draw (8.2, 0.1) -- (9,-1);
				
				\draw (-2,-.7) -- (12,-.7);

				\filldraw[rotate around = {-40.25:(6,0)}, fill = blue!15!white] (6,0) ellipse (0.3 cm and 0.2 cm);
				\filldraw[rotate around = {-40.25:(7,0)}, fill = blue!15!white] (7,0) ellipse (0.3 cm and 0.2 cm);
				\filldraw[rotate around = {-40.25:(8,0)}, fill = blue!15!white] (8,0) ellipse (0.3 cm and 0.2 cm);
				\filldraw[rotate around = {-40.25:(9,0)}, fill = blue!15!white] (9,0) ellipse (0.3 cm and 0.2 cm);
				\filldraw[rotate around = {-40.25:(10,0)}, fill = blue!15!white] (10,0) ellipse (0.3 cm and 0.2 cm);
				\filldraw[rotate around = {-40.25:(11,0)}, fill = blue!15!white] (11,0) ellipse (0.3 cm and 0.2 cm);
				
				\filldraw[rotate around = {-40.25:(6,-1)}, fill = blue!15!white] (6,-1) ellipse (0.3 cm and 0.2 cm);
				\filldraw[rotate around = {-40.25:(11,-1)}, fill = blue!15!white] (11,-1) ellipse (0.3 cm and 0.2 cm);
				
				\filldraw (7, -1) circle (0.05 cm);
				\filldraw (9,-1) circle (0.05 cm);
				\filldraw (10,-1) circle (0.05 cm);

			\end{tikzpicture}
			
			\caption{$T_3(6,15)$ (left) and $T_t^2(6,15)$ (right). Blue ellipses represent cliques of size $\frac{t-1}{2}$. Rooted vertices are below the line.}    
		\end{center}
	\end{figure}
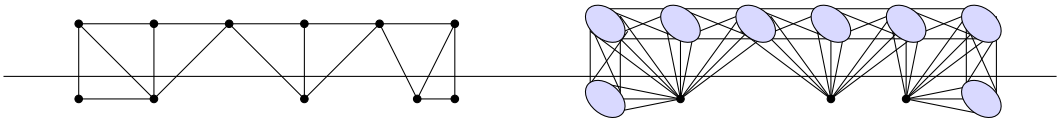

\begin{observation}\label{observation density of Tt2}
	For all odd $t \geq 3$ and positive integers $a,b$ with $a$ even and $2a + 2 \leq b$, the construction $T_t^2(a,b)$ has rooted density
	\[
	d_t^2(a,b) := \frac{a\binom{\frac{t-1}{2}}{2} + (a+1)\left(\frac{t-1}{2}\right)^2 + (b-a-1)\frac{t-1}{2}}{a\frac{t-1}{2}} = \frac{3t-9}{4} + \frac{t-3}{2a} + \frac{b}{a}.
	\]
	Moreover, $T_t^2(a,b)$ is a $K_t$-tree, and the subgraph of $T_t^2(a,b)$ induced by its root set is a disjoint union of isolated vertices and cliques of size $\frac{t+1}{2}$.
\end{observation}

\begin{proof}
		The value of $d_t^2(a,b)$ is immediate when $b \leq 3a + 1$: each of the unrooted vertices of $T_2(a,b)$ corresponds to a $\frac{t-1}{2}$-vertex clique in $T_t^2(a,b)$, contributing a total of $a \frac{t-1}{2}$ unrooted vertices and $a \binom{\frac{t-1}{2}}{2}$ unrooted edges within these cliques. Each edge $v_iv_j$ of $T_2(a,b)$ corresponds to $\left(\frac{t-1}{2}\right)^2$ edges in $T_t^2(a,b)$ if $v_i,v_j$ both correspond to $\frac{t-1}{2}$-vertex cliques in $T_t^2(a,b)$, and corresponds to $\frac{t-1}{2}$ edges in $T_t^2(a,b)$ if only one of $v_i,v_j$ corresponds to a $\frac{t-1}{2}$-vertex clique in $T_t^2(a,b)$. Thus, $T_t^2(a,b)$ contains a total of  $a \frac{t-1}{2}$ unrooted vertices and $a \binom{\frac{t-1}{2}}{2} + (a+1)\left(\frac{t-1}{2}\right)^2 + (b - a -1)\frac{t-1}{2}$ unrooted edges. 
		
		For $b > 3a + 1$, we obtain $T_t^2(a,b)$ from $T_t^2(a, b - a)$ by adding $\frac{a}{2}$ new rooted vertices, each incident to $2 \frac{t-1}{2}$ new unrooted edges, so the given formula for $d_t^2(a,b)$ inductively holds for all valid pairs $a,b$.
		
		To see that $T_t^2(a,b)$ is a $K_t$-tree, first let us handle the case when $b\leq 3a+1$. Let $v_1,v_2,\dots,v_{a+2}$ denote the vertices along the path in $T_3(a,b)$ which were replaced by cliques of size $\frac{t-1}{2}$, and for $i\in [a+1]$, let $x_i$ denote the vertex in the root of $T_3(a,b)$ which is in a triangle with $v_i$ and $v_{i+1}$. Let $K^{(i)}$ denote the copy of $K_t$ in $T_t^2(a,b)$ which contains all the vertices corresponding to $v_1,v_{i+1}$ or $x_i$. Then, we can indeed form $T_t^2(a,b)$ by gluing the $K^{(i)}$'s together in order, each clique intersects the last in either $\frac{t+1}{2}$ vertices (if the corresponding $x_i$'s are equal), or $\frac{t-1}{2}$ vertices (if the corresponding $x_i$'s are disjoint). When $b>3a+1$, $T_3(a,b)$ is formed from $T_3(a,b-a)$ by adding $a/2$ new triangles with one new rooted vertex, each such triangle corresponds to a copy of $K_t$ which intersects the clique corresponding to $v_iv_{i+1}$ and $x_i$ in exactly $t-1$ vertices. Finally, the root of $T_t^2(a,b)$ consists of two cliques of size $\frac{t+1}{2}$, one corresponding to the vertices $v_1$ and $x_1$, and one corresponding to $v_{a+2}$ and $x_{a+1}$ (note that $x_1 \neq x_{a+1}$ since $b \geq 2a+2$), along with many isolated vertices, one for each $x_i$ with $2\leq i\leq a$, and one for each additional clique glued stemming from the triangles which were glued to the spine of $T_3(a,b)$ when $b>3a+1$.
\end{proof}
	
	We again use Lemma~\ref{all or nothing lemma} to show that Construction~\ref{generalized fan construction} is balanced.
	
	\begin{prop}\label{prop:type2}
		For all odd $t \geq 3$ and positive integers $a,b$ with $a$ even and $2a + 2 \leq b$, the construction $T_t^2(a,b)$ is balanced.
	\end{prop}
	
	\begin{proof}
        It suffices to show that $T_t^2(a,b)$ is balanced for $2a+ 2 \leq b \leq 3a+ 1$. This is because for $b > 3a+1$, in $T_t^2(a,b)$ every vertex has exactly one more incident edge than $T_t^2(a,b-a)$, and these edges are to the root set. Thus the density of every set (of non-root vertices) in $T_t^2(a,b)$ is exactly one more than in $T_t^2(a,b-a)$, so the balancedness of $T_t^2(a,b)$ follows from the balancedness of $T_t^2(a,b-a)$.
        
        Let $A$ be a subset of $V(T_t^2(a,b)) \setminus R$, where $R$ is the set of root vertices of $T_t^2(a,b)$. We may also assume, as in the proof of Proposition~\ref{proposition generalized spike properties}, that $A$ induces a connected subgraph of $V(T_t^2(a,b))$, and that for $v_i \in V(T_2(a,b))\setminus R$, the corresponding clique $S_i$ of $\frac{t-1}{2}$ vertices in $V(T_t^2(a,b)) \setminus R$ is either contained in $A$ or disjoint from $A$. We define
		
		\[
		V' := \{v_i \in V(T_2(a,b))\setminus R: S_i \subseteq A\}
		\]
		and
		\[
		E' : = \{v_iv_j \in E(T_2(a,b)): v_i \in V'\},
		\]
		and set $a' := |V'|$, and $b' := |E'|$. Note that the density of $V'$ in $T_2(a,b)$ is $\frac{b'}{a'} \geq \frac{b}{a}$ since $T_2(a,b)$ is balanced. Also note that for exactly $|V'| + 1$ edges $v_iv_j \in E'$, both $v_i,v_j$ correspond to cliques of size $\frac{t-1}{2}$ in $T_t^2(a,b)$, so
		\[
		d(S) = \frac{a' \binom{\frac{t-1}{2}}{2} + (a' + 1)\left( \frac{t-1}{2} \right)^2 + (b' - a' - 1)\frac{t-1}{2} }{a'\frac{t-1}{2}} = \frac{3t - 9}{4} + \frac{t-3}{2a'} + \frac{b'}{a'},
		\]
		which is clearly at least $d_t^2(a,b) = \frac{3t-9}{4} + \frac{t-3}{2a} + \frac{b}{a}$ as $a' \leq a$ and $\frac{b'}{a'} \geq \frac{b}{a}$.
	\end{proof}
	
	Finally, we exhibit the full range of rational densities achievable by Construction~\ref{generalized fan construction}.
	
	\begin{prop}\label{prop:type2rational}
		If $t \geq 3$ is an odd integer, then for any $d \in \mathbb{Q}$ with $d>t-1$, there exist positive integers $a,b$, such that $a$ is even, $2a + 2 \leq b$, and $d_t^2(a,b) = d$.
		
	\end{prop}
	
	\begin{proof}
		Note that since $d>t-1$, we can write $d = \frac{3t - 9}{4} + \frac{x}{y}$ for some $x,y\in\mathbb{N}$ with $\frac{x}{y} > \frac{t + 5}{4} \geq 2$. Furthermore, if we set $a:= 2y$ and $b:= 2x - \frac{t-3}{2}$, then $d_t^2(a,b)=\frac{3t - 9}{4} + \frac{x}{y}$ and $a$ is both positive and even. Since $t$ is odd, $b$ is an integer, and as long as $x$ is large enough, $b$ will be positive. Finally, the condition $2a + 2 \leq b$ is equivalent to 
		\[
		\frac{x}{y} > 2 + \frac{t+1}{4y} ,
		\]
		which is fulfilled as long as $y$ is large enough. Thus, we can ensure that all the conditions claimed in the hypothesis are fulfilled by selecting a representation of $\frac{x}{y}$ with $x$ and $y$ sufficiently large. 
	\end{proof}

	\section{Exponents for cliques}\label{main thm sec}
	
	Now we compile all the ingredients to prove Theorem~\ref{thm::Ktrealizable}. We prove the following more specific theorem which is helpful for formalizing our induction.
	
	\begin{theorem}\label{thm::technical}
		For every integer $t \geq 2$, and for every rational $d \in (t/2,\infty)$, there exists balanced rooted graphs $F_1, \dots, F_k$ of rooted density at least $d$ and an integer $L_d(t)$ such that for all $L \geq L_d(t)$,
		\[ \ex\left(n,K_t,F_1^L \cup \cdots \cup F_k^L \right) = \Theta\left( n^{t-\binom{t}{2}/d} \right) ,\]
		where the implied constants may depend on any parameter but $n$.
	\end{theorem}
	
	\begin{proof}
		We proceed by induction on $t$, with the base case being the proof of Theorem~\ref{thm::BukhConlon}, which we give here.
		
		\textbf{Base case}: $t=2$. Let $d \in (1,\infty)$ be a rational number, and let $a$ and $b$ be integers such that $1 \leq a < b$ and $d = b/a$. Let $(T,R)$ be the rooted graph $T_2(a,b)$ given in Construction~\ref{constr::BC}, so by Proposition~\ref{prop::balanced:BC}, $(T,R)$ is balanced with rooted density $d$. Let $L_d(2)$ be the constant $L_0$ guaranteed by Theorem~\ref{thm::lowerbound} applied with $(F_1,R_1)=(T,R)$. Then,
		\[ 
		\ex\left(n,K_2,(T,R)^L\right) = \Omega(n^{2-1/d})
		\]
		for all $L \geq L_d(2)$.
		
		For the upper bound, let $G$ be an $n$-vertex $(T,R)^L$-free graph for $L\geq L_d(2)$. By the contrapositive of Lemma~\ref{lem::rooting} (applied with $L=L$, $\ell=1$, $s=1$ and $t=2$) $G$ has $O(n^{2-1/d})$ edges, as desired.
		
		\textbf{Inductive step}: Let $t \geq 3$ be an integer and let $d \in (t/2,\infty)$ be a rational number. Let $p:= cn^{-1/d}$ with $c$ sufficiently large to be determined later. Let $\mathcal{F}$ be the collection of all rooted graphs given by the induction hypothesis for $2 \leq i < t$. We split into two cases depending on $t$ and $d$ to determine which construction we should use.
		
		\textbf{Case 1}: $t$ is even or $d\leq t-1$. With an eye towards applying Proposition~\ref{prop::rationaltype1}, let $s = \min\{\ceil{2d-t}, \floor{t/2}\}$. If $s=t/2$, then $t$ is even, and $t/2\leq 2d+t$, which implies that $d\in \left(\frac{3t-2}{4},\infty\right)$. If $s=\frac{t-1}{2}$, then $t$ is odd, and using $s\leq \ceil{2d-t} < 2d-t+1$, we have that $d > \frac{s+t-1}{2}$, and since $d\leq t-1=\frac{t+2s-1}{2}$, we have $d\in \left( \frac{t+s-1}{2}, \frac{t + 2s - 1}{2} \right]$. Finally, if $s=\lceil 2d-t\rceil$, then since $2d-t\leq s < 2d-t+1$, we have $\frac{t+s-1}{2} < d \leq \frac{t+s}{2}$, and consequentially, $d\in \left( \frac{t+s-1}{2}, \frac{t + 2s - 1}{2} \right]$. 
		
		Thus, in all subcases we can apply Proposition~\ref{prop::rationaltype1}, so there exist positive integers $a,b$ with $a < b$ such that $d_{t}^1(a,b,s) = d$. Let $(T,R)$ be the rooted $K_t$-tree $T_t^1(a,b,s)$ described in Construction~\ref{generalized spike construction}. By Observation~\ref{observation density of dt1} and Proposition~\ref{proposition generalized spike properties}, $(T,R)$ has rooted density $d$ and is balanced since $d \leq \frac{t+2s-1}{2}$ or $s=t/2$. 
		
		\textbf{Subcase 1.1:} $d\leq t-1$. Let $s':= \ceil{2d-t}\leq t-2$, and note that $s \leq s'$. Let $(K_t,K_{s'+1})$ be the rooted complete graph on $t$ vertices with exactly $s'+1$ vertices rooted; observe that it has rooted density $\frac{\binom{t}{2}-\binom{s'+1}{2}}{t-s'-1} = \frac{t+s'}{2} \geq d$. Let $L_0$ be the constant guaranteed by Theorem~\ref{thm::lowerbound} applied with $\{(F_i,R_i)\}:=\mathcal{F}\cup\{(T,R), (K_t,K_{s'+1})\}$, and set $L_d(t):=\max\{L_0,L_d(2),\dots,L_d(t-1)\}$.
		
		We claim that for $L\geq L_d(t)$,
		\[ 
		\ex\left(n,K_t,\mathcal{F}^L \cup (T,R)^L \cup (K_t,K_{s'+1})^L \right) = \Theta(n^{t-\binom{t}{2}/d}),
		\]
            where $\mathcal{F}^L = \{ (F,R)^L : (F,R) \in \mathcal{F} \}$.
		The lower bound is provided by Theorem~\ref{thm::lowerbound}. For the upper bound, we use Lemma~\ref{lem::rooting}. Let $G$ be an $n$-vertex graph that is $\mathcal{F}^L \cup (T,R)^L \cup (K_t,K_{s'+1})^L$-free. By induction, since $L_d(t)\geq L_d(i)$, $G$ has at most $n^i p^{\binom{i}{2}}$ copies of $K_i$ for all $i < t$ for sufficiently large $c$ (recall that $p:=cn^{-1/d}$). Since $G$ is $(K_t,K_{s'+1})^L$-free, every copy of $K_{s'+1}$ is in at most $L$ copies of $K_t$ in $G$. Since via Observation~\ref{observation density of dt1}, $T[R]$ is the disjoint union of cliques on $s<t$ vertices, we also have that via induction that $G$ contains at most $n^{|R|} p^{|E(T[R])|}$ copies of $T[R]$ for sufficiently large $c$. Finally, note that by Observation~\ref{observation density of dt1}, $(T,R)$ has glue size $s \leq s'$. By the contrapositive of Lemma~\ref{lem::rooting} (applied with $L=L$, $\ell=L$, $s=s'$ and $t=t$), having checked conditions 2-4 for $G$, we must have that $G$ has at most $t n^t p^{\binom{t}{2}}$ copies of $K_t$, when $c$ and $n$ are sufficiently large, as desired.
		
		\textbf{Subcase 1.2:} $t$ is even and $d> t-1$. Let $L_0$ be the constant guaranteed by Theorem~\ref{thm::lowerbound} applied with $\{(F_i,R_i)\}:=\mathcal{F}\cup\{(T,R)\}$, and set $L_d(t):=\max\{L_0,L_d(2),\dots,L_d(t-1)\}$. Similar to the last case, we claim that for $L\geq L_d(t)$, 
		\[ 
		\ex\left(n,K_t,\mathcal{F}^L \cup (T,R)^L \cup (K_t,K_{s'+1})^L \right) = \Theta(n^{t-\binom{t}{2}/d}).
		\]
		The lower bound is again provided by Theorem~\ref{thm::lowerbound}. For the upper bound, we define $G$ to be a $\mathcal{F}^L \cup (T,R)^L$-free $n$-vertex graph, and find via induction that $G$ contains at most $n^ip^{\binom{i}{2}}$ copies of $K_i$ for all $i<t$ and at most $n^{|R|} p^{|E(T[R])|}$ copies of $T[R]$ for sufficiently large $c$. By the contrapositive of Lemma~\ref{lem::rooting} (applied with $L=L$, $\ell=1$, $s=t-1$ and $t=t$), having checked conditions 2 and 4, and noting that condition 3 is trivially satisfied when $s=t-1$, we must have that $G$ has at least $tn^tp^{\binom{t}2}$ copies of $K_t$, when $c$ and $n$ are sufficiently large.
		
		\textbf{Case 2}: $t$ is odd and $d > t-1$. By Proposition~\ref{prop:type2rational}, there exists positive integers $a,b$ such that $a$ is even, $2a+1 < b$, and $d_t^2(a,b) = d$, where $d_t^2(a,b)$ is defined in Observation~\ref{observation density of Tt2}. Let $(T,R)$ be the rooted $K_t$-tree $T_t^2(a,b)$ defined in Construction~\ref{generalized fan construction}. By Observation~\ref{observation density of Tt2} and Proposition~\ref{prop:type2}, $(T,R)$ is balanced and has rooted density at least $d$.  Let $L_0$ be the constant guaranteed by Theorem~\ref{thm::lowerbound} applied with $\{(F_i,R_i)\}:=\mathcal{F}\cup\{(T,R)\}$, and set $L_d(t):=\max\{L_0,L_d(2),\dots,L_d(t-1)\}$.
		We claim that for $L\geq L_d(t)$,
		\[ 
		\ex\left(n,K_t,\mathcal{F}^L \cup (T,R)^L \right) = \Theta(n^{t-\binom{t}{2}/d}) .
		\]
		The lower bound is provided by Theorem~\ref{thm::lowerbound}. For the upper bound, we use Lemma~\ref{lem::rooting} with $s=t-1$. Let $G$ be an $n$-vertex  graph that is $\mathcal{F}^L \cup (T,R)^L$-free. By induction, $G$ has at most $n^i p^{\binom{i}{2}}$ copies of $K_i$ for all $i<t$ for sufficiently large $c$. Trivially, every copy of $K_t$ is in at most $1$ copy of $K_t$, and $(T,R)$ has glue size at most $t-1$. Since by Observation~\ref{observation density of Tt2}, $T[R]$ is the disjoint union of cliques on less than $t$ vertices, the number of copies of $T[R]$ in $G$ is at most $n^{|R|} p^{|E(T[R])|}$ for sufficiently large $c$. By the contrapositive of Lemma~\ref{lem::rooting}, having checked conditions 2-4 of $G$, we have that $G$ has at most $t n^t p^{\binom{t}{2}}$ copies of $K_t$, when $c$ and $n$ are sufficiently large, completing the proof.
	\end{proof}

We can now quickly derive our main theorem.
	
	\begin{proof}[Proof of Theorem~\ref{thm::Ktrealizable}]
		Theorem~\ref{thm::technical} handles all rational exponents strictly between $1$ and $t$. Observe that $t$ is a realizable exponent for $K_t$, since $\ex(n,K_t,\emptyset) = \binom{n}{t}$, and $1$ is a realizable exponent for $K_t$, since $\ex(n,K_t,K_t^+) = \floor{n/t}$, where $K_t^+$ is the graph formed by adding a pendant vertex to $K_t$.
	\end{proof}

	\section{Acknowledgements}\label{sec::acknowledgements}
	
	This work was initiated and partially completed while the authors were in the University of Illinois Urbana-Champaign's combinatorics group, whose activities were supported by NSF RTG grant DMS-1937241. The second author's visit to UIUC was made possible by the National Science Foundation grant DMS-1800749.

\bibliographystyle{amsplain}
\providecommand{\bysame}{\leavevmode\hbox to3em{\hrulefill}\thinspace}
\providecommand{\MR}{\relax\ifhmode\unskip\space\fi MR }
\providecommand{\MRhref}[2]{%
  \href{http://www.ams.org/mathscinet-getitem?mr=#1}{#2}
}
\providecommand{\href}[2]{#2}

\end{document}